\documentclass[11pt]{amsart}
\usepackage{amsmath}
\usepackage{amssymb}
\usepackage{tabularx}
\usepackage{enumerate}
\usepackage{graphicx}
\usepackage{texdraw}

\usepackage{mathrsfs}

\usepackage{amsfonts,amssymb,amsmath}
\usepackage{epsfig}

\makeatletter
\@addtoreset{equation}{section}

\makeatother
\newtheorem{thm}{Theorem}[section]
\newtheorem{lem}[thm]{Lemma}

\newtheorem{prop}[thm]{Proposition}

\newcommand{\R}{\mathbb{R}}

\begin{document}
\title[Nonlinear diffusion problems with free boundaries]{Nonlinear diffusion problems with free boundaries: Convergence, transition speed and zero number arguments$^\S$}
 \thanks{$\S$ This research was partly supported by the Australian Research
Council and by the NSFC (No. 11271285). }
\author[Y. Du, B. Lou and M. Zhou]{Yihong Du$^\dag$, Bendong Lou$^\ddag$, Maolin Zhou$^\dag$}
\thanks{$\dag$ School of Science and
Technology, University of New England, Armidale, NSW 2351,
Australia.}
\thanks{$\ddag$ Department of Mathematics, Tongji
University, Shanghai 200092, China.}
\thanks{{\bf Emails:} {\sf ydu@turing.une.edu.au} (Y. Du), {\sf
blou@tongji.edu.cn} (B. Lou), {\sf zhouml@ms.u-tokyo.ac.jp} (M. Zhou)}
\date{\today}

\begin{abstract}
This paper continues the investigation of Du and Lou \cite{DLou}, where the long-time behavior of positive solutions to a nonlinear diffusion equation of the
form $u_t=u_{xx}+f(u)$ for $x$ over a varying interval $(g(t), h(t))$ was examined. Here $x=g(t)$ and $x=h(t)$ are free boundaries
evolving according to $g'(t)=-\mu u_x(t, g(t))$, $h'(t)=-\mu u_x(t,h(t))$, and $u(t, g(t))=u(t,h(t))=0$. We answer
several intriguing questions left open in \cite{DLou}.
 First we prove the conjectured convergence result in \cite{DLou} for the general case that
 $f$ is $C^1$ and $f(0)=0$. Second, 
for bistable and combustion types of $f$, we determine the asymptotic propagation speed of $h(t)$ and $g(t)$ in the transition case. More presicely, we show that when the transition case happens, for bistable type of $f$ there exists a uniquely determined $c_1>0$ such that $\lim_{t\to\infty} h(t)/\ln t=\lim_{t\to\infty} -g(t)/\ln t=c_1$, and for combustion type of $f$, there exists a uniquely determined $c_2>0$ such that $\lim_{t\to\infty} h(t)/\sqrt t=\lim_{t\to\infty} -g(t)/\sqrt t=c_2$. Our approach is based on the zero number arguments of Matano and Angenent, and on the construction of delicate upper and lower solutions.

\end{abstract}

 \maketitle

\section{Introduction}

We continue the work of Du and Lou \cite{DLou}  on certain nonlinear diffusion equations with free boundaries in space dimension 1. We are particularly interested in the long-time dynamical behavior of the problem for monostale, bistable and combustion types of nonlinearities. We answer
 several intriguing questions left open in \cite{DLou} and so  complete a rather general theory for the one space dimension case of this type of nonlinear free boundary problems.

Our nonlinear diffusion problem  has the following form:
\begin{equation}\label{p}
\left\{
\begin{array}{ll}
 u_t = u_{xx} + f(u), &  g(t)< x<h(t),\ t>0,\\
 u(t,g(t))= u(t,h(t))=0 , &  t>0,\\
g'(t)=-\mu\, u_x(t, g(t)), & t>0,\\
 h'(t) = -\mu \, u_x (t, h(t)) , & t>0,\\
-g(0)=h(0)= h_0,\ \ u(0,x) =u_0 (x),& -h_0\leq x \leq h_0,
\end{array}
\right.
\end{equation}
where $x=g(t)$ and $x=h(t)$ are the moving boundaries to be
determined together with $u(t,x)$, $\mu$ is a given positive
constant, $f:[0,\infty)\rightarrow \R$ is a $C^1$ function
satisfying
\begin{equation}\label{f(0)}
f(0)=0.
\end{equation}
The initial function $u_0$ belongs to  $ \mathscr {X}(h_0)$ for some
$h_0>0$, where
\begin{equation}\label{def:X}
\begin{array}{ll}
\mathscr {X}(h_0):= \Big\{ \phi \in C^2 ([-h_0,h_0]): & \phi(-h_0)=
\phi (h_0)=0,\; \phi'(-h_0)>0,\\ & \phi'(h_0)<0,\;
 \phi(x) >0 \ \mbox{in } (-h_0,h_0)\;\Big\}.
\end{array}
\end{equation}

Under these general conditions, \eqref{p} has a unique locally defined classical solution, which is globally defined if $u(t,x)$ stays finite for every $t>0$.
In particular the solution is globally defined if
there exists $C>0$ such that $u(t,x)\leq C$ whenever it is defined. Such an a priori bound of the solution is guaranteed if we assume further that
$f(u)\leq 0$ for all large $u$, say for $u\geq M$ with some $M>0$. Moreover, $g'(t)<0$ and $h'(t)>0$ as long as they are defined. Therefore,
in the case that $(u, g, h)$ is defined for all $t>0$, $g_\infty:=\lim_{t\to\infty} g(t)$ and $h_\infty:=\lim_{t\to\infty} h(t)$ are well-defined.

The first main result of \cite{DLou} is the following convergence theorem for a general nonlinear term, namely $f$ is $C^1$ satisfying $f(0)=0$.

\smallskip

\noindent
{\bf Theorem A.} {\it Under the above assumptions on $f$, suppose that  $(u,
g,h)$ is a solution of \eqref{p} that is defined for all $t>0$, and
$u(t,x)$ stays bounded, namely
\[
u(t,x)\leq C \mbox{ for all $t>0$, \ $x\in [g(t), h(t)]$ and some
$C>0$.}
\]
Then $(g_\infty, h_\infty)$ is either a finite interval or
$(g_\infty, h_\infty)=\R^1$. Moreover, if $(g_\infty, h_\infty)$ is
 a finite interval, then
$\lim_{t\to\infty}\max_{x\in[g(t),h(t)]} u(t,x)=0$, and if $(g_\infty, h_\infty)=\R^1$
then either $\lim_{t\to\infty} u(t,x)$ is a nonnegative constant
solution of
\begin{equation}
\label{ellip}
 v_{xx}+f(v)=0,\; x\in \R^1,
 \end{equation}
 or
 \begin{equation}
 \label{gamma(t)}
 u(t,x)-v(x+\gamma(t))\to 0 \mbox{ as } t\to\infty,
 \end{equation}
 where $v$ is an evenly decreasing positive solution
 of \eqref{ellip},  $\gamma: [0,\infty)\to [-h_0,h_0]$ is a
 continuous function, and the convergence of $u$ as $t\to\infty$ is uniform over any bounded interval of $x$.}

\smallskip

Here we say $v(x)$ is evenly decreasing if $v$ is an even function and $v'(x)<0$ for $x>0$.
\smallskip

When \eqref{gamma(t)} holds, it is conjectured in \cite{DLou} that $\lim_{t\to\infty} \gamma(t)$ exists. 
Our first theorem in this paper gives a
positive answer to this conjecture.

\begin{thm}
\label{conv}
If \eqref{gamma(t)} holds in Theorem A,  then $\lim_{t\to\infty} \gamma(t)=x_0$
for some $x_0\in [-h_0,h_0]$.
 Therefore we have
\[
u(t,x)-v(x+x_0)\to 0  \mbox{ as } t\to\infty.
\]
\end{thm}

For monostable, bistable and combustion types of $f(u)$ (to be recalled in detail below), \cite{DLou} examined the long-time behavior of
$(u,g,h)$. If $f(u)$ is monostable, it is shown that there is a spreading-vanishing dichotomy:

\smallskip

\noindent
{\bf Theorem B.} {\it Suppose that $f(u)$ is monostable. Then
the solution $(u,g,h)$ is defined globally and as $t\to\infty$, either
\begin{itemize}
\item[(i)] {\bf Spreading:} $(g_\infty,h_\infty)=\R^1$ and $\lim_{t\to\infty} u(t,x)=1$ locally uniformly in $\R^1$,
\\
 or
\item[(ii)]{\bf Vanishing:} $(g_\infty, h_\infty)$ is a finite interval with length no bigger than $\pi/\sqrt{f'(0)}$ and
$\lim_{t\to\infty}\max_{g(t)\leq x\leq h(t)}u(t,x)=0$.
\end{itemize}
}
\smallskip

In contrast, for bistable and combustion types of $f(u)$, a trichotomy holds:

\smallskip

\noindent
{\bf Theorem C.}
{\it
 If $f(u)$ is bistable, then the solution $(u,g,h)$ is defined globally and as $t\to\infty$, either
\begin{itemize}
\item[(i)] {\rm \bf Spreading:} $(g_\infty, h_\infty)=\R^1$ and
$
\lim_{t\to\infty}u(t,x)=1$  locally uniformly in $\R^1$,\\
or

\item[(ii)] {\rm \bf Vanishing:} $(g_\infty, h_\infty)$ is a finite interval
and
$
\lim_{t\to\infty}\max_{g(t)\leq x\leq h(t)} u(t,x)=0,
$\\
or

\item[(iii)] {\rm \bf Transition:} $(g_\infty, h_\infty)=\R^1$ and there exists
a continuous function $\gamma: [0,\infty)\to [-h_0,h_0]$ such that \footnote{By Theorem 1.1, the conclusion here can now be improved to: There exists $x_0\in [-h_0,h_0]$ such that $u(t,x)-V(x+x_0)\to 0$ in $L^\infty_{loc}(\R^1)$ as $t\to \infty$.}
\[
\lim_{t\to\infty}|u(t,x)- V (x+\gamma(t))|=0 \mbox{ locally
uniformly in $\R^1$},
\]
where $V$ is the unique ground state, that is, the unique positive solution to
\[v''+f(v)=0 \; (x\in\R^1),\; v'(0)=0,\; v(-\infty)=v(+\infty)=0.\]
\end{itemize}
}

\smallskip

\noindent
{\bf Theorem D.}
{\it If $f(u)$ is of combustion type, then the solution $(u,g,h)$ is defined globally and as $t\to\infty$, either
\begin{itemize}
\item[(i)] {\rm \bf Spreading:} $(g_\infty, h_\infty)=\R^1$ and
$
\lim_{t\to\infty}u(t,x)=1$ locally uniformly in $\R^1$,
\\
or

\item[(ii)]{\rm \bf Vanishing:} $(g_\infty, h_\infty)$ is a finite interval
and
$
\lim_{t\to\infty}\max_{g(t)\leq x\leq h(t)} u(t,x)=0,
$\\
or

\item[(iii)] {\rm \bf Transition:} $(g_\infty, h_\infty)=\R^1$ and
$
\lim_{t\to\infty}u(t,x)=\theta$ locally uniformly in $\R^1$, where $\theta$ is the largest zero of $f(u)$ in $(0,1)$.
\end{itemize}
}
\smallskip

If we take the initial function of the form $u_0=\sigma \phi$ for some $\phi\in \mathscr {X}(h_0)$,
it is shown in \cite{DLou} that in Theorems C and D, there exists $\sigma^* = \sigma^* (h_0, \phi)\in (0,\infty]$
such that vanishing happens when $ 0<\sigma < \sigma^*$, spreading
happens when $\sigma>\sigma^*$, and transition happens when $\sigma=\sigma^*$.

When spreading happens, the following result of \cite{DLou} gives a first estimate of the spreading speed.

\smallskip

\noindent
{\bf Theorem E.} {\it  Suppose that $f(u)$ is of monostable, bistable or combustion type. Then the problem
\begin{equation}\label{prop-profile}
\left\{
  \begin{array}{l}
  q_{zz} - c q_z + f(q) =0\ \ \mbox{ for }  z\in (0,\infty),\\
  q(0)=0, \; \mu q_z(0) = c,\; q(\infty)=1,\; q(z)>0 \mbox{ for } z>0.
  \end{array}
  \right.
\end{equation}
has a unique solution pair $(c,q)=(c^*, q^*)$, and $c^*>0$, $(q^*)'(z)>0$.
Moreover, if spreading happens in Theorems B, C or D, then
\[
\lim_{t\to\infty}\frac{h(t)}{t}=\lim_{t\to\infty}\frac{-g(t)}{t}=c^*.
\]
}

What is missing from \cite{DLou} is an estimate of the propagation speed of $h(t)$ and $g(t)$ in the transition cases of Theorems C and D. This turns out to be a difficult mathematical question, especially for the combustion case. Our second main result in this paper gives a first estimate of the propagation speed for these transition cases.

In order to state these estimates precisely, we recall that $f$ is called {\bf bistable},  if $f\in C^1$ and it
satisfies
\begin{equation}\label{bi}
f(0)=f(\theta)= f(1)=0, \quad f(u) \left\{
\begin{array}{l}
<0 \ \ \mbox{in } (0,\theta),\\
>0\ \  \mbox{in } (\theta, 1),\\
< 0\ \ \mbox{in } (1,\infty)
\end{array} \right.
\end{equation}
for some $\theta\in (0,1)$,  $f'(0)<0$, $f'(1)<0$ and
\begin{equation}\label{unbalance}
\int_0^1 f(s) ds >0.
\end{equation}
We say $f$ is of {\bf combustion type},  if $f\in C^1$ and it
satisfies
\begin{equation}\label{combus}
f(u)=0 \ \ \mbox{in } [0,\theta], \quad f(u) >0 \ \mbox{in }
(\theta,1), \quad f'(1)<0,\quad f(u) < 0 \ \mbox{in } (1, \infty)
\end{equation}
for some $\theta \in (0,1)$, and there exists a small $\delta  >0$
such that $f(u)$ is nondecreasing in $(\theta, \theta+\delta)$.

\begin{thm}
\label{transition} Suppose additionally \footnote{In the combustion case, $f(u)\equiv 0$ in $[0,\theta]$, and hence \eqref{1+alpha} is automatically satisfied.}
\begin{equation}
\label{1+alpha}
f\in C^{1+\alpha}([0,\delta]) \mbox{ for some small $\delta>0$ and some $\alpha\in (0,1)$}.
\end{equation}
Then in the transition case of Theorem C, we have
\[\mbox{
$h(t), -g(t) =\lambda_0\ln t + O(1)$ with $\lambda_0 =[-f'(0)]^{-1/2}$,}
 \]
 and in the transition case of Theorem D, we have
 \[
 h(t), -g(t)=2\xi_0\sqrt{t} \,[1+o(1)],
  \]
   where $\xi_0>0$ is uniquely determined by
\begin{equation}\label{def-xi_0}
2 \xi_0 e^{\xi^2_0} \int _0^{\xi_0}e^{-s^2}ds=\mu \theta.
\end{equation}
\end{thm}

Free boundary problems of the form \eqref{p} was first studied in \cite{DLin} for the special case $f(u)=au-bu^2$.
When $f(u)\equiv 0$, \eqref{p} reduces to the classical Stefan problem describing the melting of ice in contact with water (in a simplified one space dimension setting).
In such a situation, $u(t,x)$ represents the temperature of water, and the free boundaries are the ice-water interphases.
Problem \eqref{p} with a nonlinear $f(u)$ may arise if one considers the situation that water is replaced by a heat conductive and chemically reactive liquid, where $f(u)$ governs the reaction. The study of \cite{DLin}, however, was motivated by
investigation of the spreading of a new or invasive species, where
the free boundaries $x=g(t)$ and $x=h(t)$ represent the spreading fronts of
the species whose density is $u(t,x)$.  Together with \cite{DLou}, the current paper provides a rather complete understanding of
the dynamics of \eqref{p} in one space dimension.
The high space dimension versions of \eqref{p}
was considered in \cite{DG1, DG2, DMW, DMZ2}, but the theory for this more challenging situation is not as  complete yet compared with the theory for the one space dimension case established in \cite{DLou} and here.

One main ingredient in our approach here is the zero number arguments of Matano and Angenent. The zero number argument was first introduced by Matano \cite{Matano} to  prove some important convergence results for nonlinear parabolic equations over bounded spatial intervals,  and it was  further developed by Angenent \cite{Ang} and others. It has proven to be a very powerful tool for treating parabolic equations in one space dimension, with  several new applications
found recently  (see, for example, \cite{DLou, DM, DGM, P, PY}). Our application of the zero number argument here (especially in Section 4) provides one more example, but with a rather different nature.

We would like to remark that, the  estimate in Theorem E for the spreading speed has been sharpened recently. In \cite{DMZ} it is proved that
when spreading happens in Theorems B, C or D,  there exist $h^0, g^0\in\R^1$ (depending on $f$ and the initial conditions) such that, as $t\to\infty$,
$$
|h(t)-c^* t-h^0| \to 0,\ \ |g(t)+c^* t+g^0|\to 0,\ \ h'(t)\to c^*,\; g'(t)\to -c^*
$$
and
\[
\max_{0\leq x\leq h(t)} |u(t,x)-q^*(h(t)-x)|\to 0, \; \max_{g(t)\leq x\leq 0} |u(t,x)-q^*(x-g(t))|\to 0.
\]
However, it appears unliekly that the techniques in this paper can be modified to prove similar sharper result for the transition case.

The rest of this paper is organized as follows. In Section 2, we prove Theorem 1.1 by using and extending the zero number argument of Angenent \cite{Ang}. In Section 3, we prove Theorem 1.2 for the bistable case, by constructing suitable upper and lower solutions. Section 4 is technically the most challenging part of the paper, where  we prove Theorem 1.2 for the combustion case; here  we make use of the zero number arguments again to handle several key steps of the proof.

\section{Zero Number Arguments and Convergence }\label{sec:convergence}

In this section we make use of the zero number arguments to prove Theorem \ref{conv}.
The following lemma is an easy consequence of the proofs of Theorems C and D of Angenent \cite{Ang}, which is the starting point of our zero number arguments.

\begin{lem}\label{angenent}
Let $u:[0, T]\times [0,1]\to \R^1$ be a bounded classical solution of
\begin{equation}
\label{linear}
u_t=a(t,x)u_{xx}+b(t,x)u_x+c(t,x)u
\end{equation}
with boundary conditions
\[
u(t,0)=l_0(t), \;u(t, 1)=l_1(t),
\]
 where $l_0, l_1\in C^1([0,T])$, and each function is either identically zero or never zero for $t\in [0,T]$. Suppose that
\[
a, 1/a, a_t, a_x, a_{xx}, b, b_t, b_x, c \in L^\infty, \mbox{ and } u(0,\cdot)\not\equiv 0 \mbox{ when $l_0=l_1\equiv 0$}.
\]
 Then for each $t\in (0, T]$, the number of zeros of $u(t,\cdot)$  in $[0, 1]$ is finite, which will be denoted by $z(t)$.
Moreover, $z(t)$ is nonincreasing in $t$ for $t\in (0, T]$, and if for some $t_0\in(0, T]$ the function $u(t_0,\cdot)$ has a degenerate zero $x_0\in [0,1]$, then $z(t_1)>z(t_2)$ for all $t_1, t_2\in (0,T]$ satisfying $t_1<t_0<t_2$.
\end{lem}

For convenience of applications later we give a
 variant of Lemma \ref{angenent}.

\begin{lem}\label{zero-number}
 Let $\xi(t)<\eta(t)$ be two continuous functions for $t\in (t_0, t_1)$. If $u(t,x)$ is a continuous function
 for $t\in (t_0, t_1)$ and $x\in [\xi(t),\eta(t)]$, and satisfies
   \eqref{linear} in the classical sense for such $(t,x)$, with
\[
u(t,\xi(t))\not=0,\; u(t, \eta(t))\not=0 \mbox{ for } t\in (t_0, t_1),
\]
then for each $t\in (t_0, t_1)$, the number of zeros of $u(t,\cdot)$  in $[\xi(t), \eta(t)]$ is finite, which we denote by $Z(t)$.
Moreover ${Z}(t)$ is nonincreasing in $t$ for $t\in (t_0, t_1)$, and if for some $s\in (t_0, t_1)$ the function $u(s,\cdot)$ has a degenerate zero $x_0\in
(\xi(t), \eta(t))$, then ${Z}(s_1)>{Z}(s_2)$ for all $s_1, s_2$ satisfying $t_0<s_1<s<s_2<t_1$.
\end{lem}
\begin{proof}
For any given $t^*\in (t_0, t_1)$, we can find $\epsilon>0$ and $\delta>0$ small such that $u(t,x)\not=0$ for $t\in I_{t^*}:=[t^*-\delta, t^*+\delta]\subset (t_0, t_1)$
and $x\in [\xi(t),\xi(t^*)+\epsilon]\cup [ \eta(t^*)-\epsilon, \eta(t)]$. Hence we may apply Lemma \ref{angenent} with $[0,T]\times [0,1]$ replaced by
$[t^*-\delta, t^*+\delta]\times [\xi(t^*)+\epsilon, \eta(t^*)-\epsilon]$ to see that the conclusions for ${Z}(t)$ hold for $t\in I_{t^*}$. Since any compact subinterval of $(t_0, t_1)$ can be covered by finitely many such $I_{t^*}$, we see that ${Z}(t)$ has the required properties over any compact subinterval of $(t_0, t_1)$. It follows that ${Z}(t)$ has the required properties for $t\in (t_0, t_1)$.
\end{proof}

Next we make use of Lemma \ref{zero-number} and a  result of Fernandez \cite{F} to prove  Theorem \ref{conv}.
We first prove a zero number conclusion.
Let $(u,g, h)$ be a solution of \eqref{p} that is defined for all $t>0$.
Denote $k(t):= \min\{h(t),-g(t)\}$ and
$$
w(t,x) := u(t,x)- u(t,-x), \quad x\in I(t) := [-k(t), k(t)],\ t>0.
$$
Let $\mathcal{Z}(t)$ be the number of zeros of the function $w(t,\cdot)$
in the closed interval $I(t)$. We notice that $w$ satisfies
\[
w_t=w_{xx} +c(t,x) w \mbox{ for } x\in (-k(t), k(t)),\; t>0,
\]
with $c(t,x):=[f(u(t,x))-f(u(t,-x))]/w(t,x)$ when $w(t,x)\not=0$, and $c(t,x)=0$ otherwise.

\begin{lem}\label{lem2.3}
Suppose that $k(t)\not\equiv K(t):=\max\{h(t), -g(t)\}$ for $t\in (0,+\infty)$. Then either
\begin{itemize}
\item[(i)]
  $w(t,x)\equiv 0$ for all large $t$,
or
\item[(ii)] there exists $t_0>0$ such that
 $\mathcal{Z}(t)$ is finite and nonincreasing in $t$ for $t> t_0$, and
 if  $w(s,\cdot)$ has a degenerate zero in the interior of $I(s)$ for some $s>t_0$, then
 $\mathcal{Z}(s_1)>\mathcal{Z}(s_2)$ for any $s_1$ and $s_2$ satisfying $t_0<s_1<s<s_2$.
\end{itemize}

 \end{lem}

Lemma 2.3 will follow from Lemma 2.2 and the following result.

\begin{lem}\label{lem2.4}
Suppose that $0<t_0<t_1<+\infty$ and
\[
k(t)<K(t) \mbox{ for } t\in [t_0,t_1),\; k(t_1)=K(t_1).
\]
Then either 
\begin{itemize}
\item[(i)] $k(t)\equiv K(t)$ and $w(t,x)\equiv 0$ for $t\geq t_1$,
or
\item[(ii)] there exists $s_0\in (t_0, t_1)$ and $s_1>t_1$ such that  $k(t)<K(t)$ for $t\in(t_1, s_1]$, and
$\mathcal{Z}(t)$ has the properties described in Lemma 2.3 case {\rm (ii)} for $t\in (t_0, s_1]$, with
\begin{equation}
\label{t1}
\mathcal{Z}(t)\equiv \mathcal{Z}(s_0)\geq \mathcal{Z}(t_1) \mbox{ for } t\in [s_0,t_1),\; 
\mathcal{Z}(t)\equiv \mathcal{Z}(s_1)\leq \mathcal{Z}(t_1)-2 \mbox{ for } t\in (t_1, s_1].
\end{equation}
\end{itemize}
\end{lem}

\begin{proof} Suppose that alternative (i) does not happen. We show that the conclusions in (ii) hold.

By Lemma 2.2,  $\mathcal{Z}(t)$ has the properties described in case (ii) of Lemma 2.3 for $t\in (t_0, t_1)$; namely it is finite and nonincreasing for
 $t\in(t_0,t_1)$, and each time a degenerate zero appears for $w(t,\cdot)$ the value of
 $\mathcal{Z}(t)$ is decreased by at least 1. These facts imply that in the interval $(t_0, t_1)$ there can exist at most finitely many values of $t$
 such that $w(t,\cdot)$ has a degenerate zero. Thus we can find $s_0\in (t_0, t_1)$ such that for $t\in [s_0, t_1)$, $w(t,\cdot)$ has only nondegenerate zeros in $I(t)$.
 Clearly $w(t,0)=0$ so $x=0$ is always a zero of $w(t,\cdot)$. Due to the non-degeneracy, the zeros of $w(t,\cdot)$, with $t\in [s_0, t_1)$, can be expressed as smooth curves:
 \[
 \mbox{$x=\gamma_1(t),..., x=\gamma_m(t)$, with $-k(t)<\gamma_i(t)<\gamma_{i+1}(t)<k(t)$ for $i=1,..., m-1$.}
 \]

 For each $i\in\{1,..., m\}$, we now examine the limit of $\gamma_i(t)$ as $t\nearrow t_1$. Clearly
 \[
 \mbox{ $x_j:=\liminf_{t\nearrow t_1}\gamma_i(t)\geq -k(t_1)$ and $x_j^*:=\limsup_{t\nearrow t_1}\gamma_i(t)\leq k(t_1)$.}
  \]
  If $x_j<x_j^*$, then it is easily seen that $w(t_1, x)\equiv 0$ for $x\in [x_j, x_j^*]$.
  We may now apply Theorem 2 of \cite{F} to $w$ over the region $[t_1-\epsilon, t_1]\times [-k(t_1-\epsilon), k(t_1-\epsilon)]$,
  with $\epsilon>0$ sufficiently small, to conclude that $w(t_1, x)\equiv 0$ for $x\in  [-k(t_1-\epsilon), k(t_1-\epsilon)]$.
  Letting $\epsilon\to 0$ we deduce $w(t_1, x)\equiv 0$ for $x\in [-k(t_1), k(t_1)]$. This implies that $u(t_1, x)$ is even in $x$. Since $g(t_1)=-h(t_1)$ and
  $u(t_1, g(t_1))=u(t_1, h(t_1))=0$, by the uniqueness of the solution to the  free boundary problem \eqref{p} (with initial time $t_1$) we deduce that $u(t,\cdot)$ is even, and $g(t)=-h(t)$ for all $t\geq t_1$.
  But this contradicts our assumption that case (i) does not happen.
   Therefore $x_j:=\lim_{t\nearrow t_1}\gamma_j(t)$ exists for every $j\in \{1,..., m\}$.

  \noindent
  {\bf Claim 1:} $x_1=-k(t_1)$ and $x_m=k(t_1)$.

  We only prove $x_m=k(t_1)$; the proof for $x_1=-k(t_1)$ is done similarly. Arguing indirectly we assume that $x_m<k(t_1)$. Then in the
  region $A_m:=\{(t,x): \gamma_m(t)<x<k(t), s_0<t\leq t_1\}$, by the maximum principle, we have $w(t,x)>0$. Since $w(t_1, k(t_1))=0$, we can apply the Hopf boundary lemma (see, e.g., Lemma 2.6 of \cite{L}) to deduce that
  $w_x(t_1, k(t_1))<0$. It follows that $u_x(t_1, k(t_1))+u_x(t_1, -k(t_1))<0$. Since $k(t_1)=h(t_1)=-g(t_1)$, we thus obtain
  \[
  h'(t_1)=-\mu u_x(t_1, h(t_1))>\mu u_x(t_1, g(t_1))=-g'(t_1).
  \]
  On the other hand, from $-g(t)<h(t)$ for $t\in [s_0, t_1)$ and $h(t_1)=-g(t_1)$ we deduce $h'(t_1)\leq -g'(t_1)$. This contradiction completes our proof of Claim 1.

  \noindent
  {\bf Claim 2:} If $x_i<x_{i+1}$, then $w(t_1,x)\not=0$ for $x\in (x_i, x_{i+1})$.

  This follows directly from the strong maximum principle applied to the region $A_i:=\{(t,x): \gamma_i(t)<x<\gamma_{i+1}(t),\; s_0\leq t\leq t_1\}$.

  \smallskip

  From Claims 1 and 2, we immediately see that $\mathcal{Z}(t_1)\leq m=\mathcal{Z}(t)$ for $t\in [s_0, t_1)$. Let $-k(t_1)=z_1<z_2<...<z_n=k(t_1)$ denote all the zeros of $w(t_1,\cdot)$ in $I(t_1)$ (with $n\leq m$).
  \smallskip

  \noindent
  {\bf Claim 3:} Denote $z^*=(z_{n-1}+z_n)/2$. There exists $\epsilon>0$ such that
  $w(t,x)\not=0$ for $t\in (t_1, t_1+\epsilon)$ and $x\in [z^*, k(t)]$.

  Clearly $w(t_1, z^*)\not=0$. For definiteness, we assume that $w(t_1, z^*)>0$. Hence by continuity there exists $\epsilon>0$ such that
  $w(t, z^*)>0$ for $t\in[t_1, t_1+\epsilon]$. We now consider $u(t,x)$ and $v(t,x):=u(t,-x)$. Since $u(t_1, x)>v(t_1, x)$ for $x\in [z^*, k(t_1))$, and
  $u(t,z^*)>v(t, z^*)$ for $t\in [t_1, t_1+\epsilon]$, and $h(t_1)=-g(t_1)$, we find that the comparison principle
   (see, e.g. Lemma 2.2 of \cite{DLou}) can be used to deduce that $-g(t)\leq h(t)$ for $t\in (t_1, t_1+\epsilon]$ and $v(t,x)\leq u(t,x)$ for $t\in (t_1, t_1+\epsilon]$ and $x\in [z^*, -g(t)]$. We may use the strong maximum principle to deduce that $v(t,x)<u(t,x)$ for $t\in (t_1, t_1+\epsilon)$ and $x\in [z^*, -g(t))$.
  We can further show that $-g(t)<h(t)$ for $t\in (t_1, t_1+\epsilon]$, since if $-g(t^*)=h(t^*)=x^*$ for some $t^*\in (t_1, t_1+\epsilon]$, then
  necessarily $w(t^*,x^*)=0$ and we can apply the Hopf lemma to deduce $w_x(t^*,x^*)<0$, which implies $-g'(t^*)<h'(t^*)$, a contradiction.
  Thus we have $k(t)=-g(t)$ for $t\in (t_1, t_1+\epsilon)$ and $w(t,k(t))>0$ for such $t$. Hence
  $w(t,x)>0$ in $\{(t,x): z^*\leq x\leq k(t), t_1\leq t\leq t_1+\epsilon\}\setminus\{(t_1, k(t_1))\}$.

  \smallskip

  \noindent
  {\bf Claim 4:} There exists $s_1>t_1$ such that $\mathcal{Z}(t_1)-2\geq p:=\mathcal{Z}(t)$ for $t\in (t_1, s_1]$.

  By Claim 3 we have $w(t,-k(t))=-w(t, k(t))\not=0$ for $t\in (t_1, t_1+\epsilon]$. Moreover, we can find $\epsilon_1>0$ very small and
  a continuous function $\tilde k(t)$ defined over $J:=[t_1-\epsilon_1, t_1+\epsilon]$ such that
  \begin{equation}
  \label{tilde-k}
  \tilde k(t)< k(t) \mbox{ and } w(t, \tilde k(t))\not=0 \mbox{ in } J, \; w(t, x)\not=0 \mbox{ for }
  x\in [\tilde k(t), k(t)],\; t\in (t_1, t_1+\epsilon].
  \end{equation}
  Since $\mathcal{Z}(t_1)$ is finite, this allows us to apply Lemma \ref{zero-number}  to conclude that,
   there exists $s_1\in(t_1, t_1+\epsilon]$ such that $w(t,\cdot)$ has no degenerate zeros in $I(t)$ when $t\in (t_1, s_1]$.
  Let $\tilde\gamma_1(t)<\tilde\gamma_2(t)<...<\tilde\gamma_p(t)$ be the nondegenerate zeros of $w(t,\cdot)$ in $I(t)$, with $t\in (t_1, s_1]$.
  Then $x=\tilde \gamma_i(t) \; (i=1,..., p)$ are smooth curves. Moreover, $\tilde z_i:=\lim_{t\searrow t_1}\tilde \gamma_i(t)$ exists for each $i\in\{1,..., p\}$, for otherwise
  $w(t_1,\cdot)$ would be identically zero over some interval of $x$, contradicting to what is known about $w(t_1,\cdot)$.
   Furthermore, $\tilde z_i<\tilde z_{i+1}$ for $i\in\{1,..., p-1\}$, since otherwise,
  we may apply the maximum principle over the region $\tilde A_i:=\{(t,x): \tilde \gamma_i(t)<x<\tilde\gamma_{i+1}(t),\; t_1\leq t\leq s_1\}$ to deduce that $w\equiv 0$ in $\tilde A_i$. Finally from \eqref{tilde-k} we know that none of these curves $\{(t,\tilde \gamma_i(t))\}\subset \R^2$  can connect to the point $(t_1,-k(t_1))$ or $(t_1, k(t_1))$.
  Thus $\tilde z_1<\tilde z_2<...<\tilde z_p$ are different zeros of $w(t_1, \cdot)$ in $I(t_1)\setminus\{-k(t_1), k(t_1)\}$. It follows immediately that $p\leq n-2$. Claim 4 is proved.

We have now proved \eqref{t1}, which shows that $\mathcal{Z}(t)$ has the properties described in Lemma 2.3 case (ii) for $t\in [s_0, s_1]$. We already know that $\mathcal{Z}(t)$ has these properties for $t\in (t_0, t_1)$. Therefore it has these properties for all $t\in (t_0, s_1]$.
\end{proof}

 \begin{proof}[Proof of Lemma 2.3]
 Since $k(t)\not\equiv K(t)$ in $(0,\infty)$, we can find $t_0>0$ such that $k(t_0)<K(t_0)$.
 Therefore there exists $t_1\in (t_0,+\infty]$ such that $k(t)<K(t)$ for $t\in [t_0, t_1)$, and $k(t_1)=K(t_1)$ when $t_1$ is finite.
Without loss of generality we assume that
$k(t_0)=-g(t_0)$ and $K(t_0)=h(t_0)$. Then necessarily $k(t)=-g(t)<h(t)=K(t)$ for all $t\in [t_0, t_1)$. It follows that
$w(t, -k(t))<0<w(t, k(t))$ for $t\in [t_0, t_1)$. Hence we can apply Lemma \ref{zero-number} to see that $\mathcal{Z}(t)$
has the required properties for $t\in (t_0, t_1)$.

Suppose that case (i) does not happen. We prove that (ii) holds.
If $t_1=+\infty$, then the proof is complete.
Suppose next that $t_1<+\infty$. By Lemma 2.4 there exists $s_1>t_1$ such that $k(t)<K(t)$ for $t\in (t_1, s_1]$ and $\mathcal{Z}(t)$ has the required properties for $t\in (t_0, s_1]$, with 
\[
\mbox{ $\mathcal{Z}(t)\equiv \mathcal{Z}(s_1)
\leq \mathcal{Z}(t_1)-2 $ for $t\in (t_1, s_1]$.}
\]

If $t_1$ is the last zero of $K(t)-k(t)$, then $w(t, -k(t))=-w(t, k(t))\not=0$ for $t>t_1$, and we can use Lemma 2.2 to conclude that $\mathcal{Z}(t)$ has the required properties for $t>t_1$. Thus in this case $\mathcal{Z}(t)$ has the required properties for all $t>t_0$, and the proof is complete.

If $t_1$ is not the last zero of $K(t)-k(t)$, then 
  there exists $t_2>s_1$, which is the first zero of $K(t)-k(t)$ after $t_1$. We may now apply Lemma 2.4 with
  $\{t_0, t_1\}$ replaced by $\{s_1, t_2\}$ to conclude that, there exists  $s_2>t_2$ such that $\mathcal{Z}(t)$ has the required properties for $t\in (t_1, s_2]$ with
\[
\mathcal{Z}(t)\equiv \mathcal{Z}(s_2)\leq \mathcal{Z}(t_2)-2 \leq \mathcal{Z}(t_1)-4 \mbox{ for } t\in (t_2, s_2].
\]
Since we already know that $\mathcal{Z}(t)$ has the required properties for $t\in (t_0, s_1]$, we find that $\mathcal{Z}(t)$ has the required properties for all $t\in (t_0, s_2]$.

If $t_2$ is the last zero of $K(t)-k(t)$ then as before we easily see that $\mathcal{Z}(t)$ has the required properties for all $t>t_0$. Otherwise we can repeat the analysis to find $s_3>t_3>s_2$  such that $t_3$ is the first zero of $K(t)-k(t)$ after $t_2$, and $\mathcal{Z}(t)$ has the required properties for $t\in (t_0, s_3]$ with
\[
\mathcal{Z}(t)\equiv \mathcal{Z}(s_3)\leq \mathcal{Z}(t_3)-2 \leq\mathcal{Z}(t_2)-4\leq  \mathcal{Z}(t_1)-6 \mbox{ for } t\in (t_3, s_3].
\]
 Since $\mathcal{Z}(t_1)$ is finite, the above process can continue only finitely many steps, say $K(t)-k(t)$ has consecutive zeros  $t_1<t_2<... <t_k$,  $t_k$ being the last zero of $K(t)-k(t)$, $\mathcal{Z}(t_k)\leq \mathcal{Z}(t_1)-2k$, and $\mathcal{Z}(t)$ has the required properties for all $t>t_0$. The proof is complete.
 \end{proof}

\smallskip

\noindent
{\bf Proof of Theorem \ref{conv}:}
Suppose by way of contradiction that $\lim_{t\to\infty}\gamma(t)$ does not exist. Then
\begin{equation}\label{assu-gamma}
-h_0\leq \liminf\limits_{t\to \infty} \gamma(t)  <\limsup\limits_{t\to \infty} \gamma(t)\leq h_0.
\end{equation}
By standard parabolic regularity we have
$$
\| u(t,\cdot -\gamma(t)) -v(\cdot )\|_{C^2 (J)} \to 0\quad \mbox{as } t\to \infty,
$$
where $J:= [-3h_0, 3h_0]$. So there exists $T_1 >0$ such that for $t>T_1$, $u(t,\cdot)$ has
exactly one maximum point $x(t) \thickapprox -\gamma(t)$ on $J$, and $x(t)$ is a continuous function of  $t$. Fix
\[
x_0\in (-\limsup_{t\to \infty} \gamma(t), -\liminf_{t\to \infty} \gamma(t)).
\]
By our assumption \eqref{assu-gamma},
$x(t)-x_0$ changes sign infinitely many times as $t$ goes to infinity. Therefore there is a sequence $t_k\to+\infty$ such that $x(t_k)=x_0$.

We now define $\tilde u(t,x)=u(t, x_0+x)$, $\tilde g(t)=g(t)-x_0$ and $\tilde h(t)=h(t)-x_0$. By perturbing $x_0$ if necessary we can always guarantee that
\[
\tilde k(t):=\min\{\tilde h(t), -\tilde g(t)\}\not\equiv \tilde K(t):=\max\{\tilde h(t), -\tilde g(t)\}.
\]
We may apply Lemma \ref{lem2.3} to $(\tilde u, \tilde g, \tilde h)$ to obtain the conclusions in case (ii) there for the zero number $\tilde{\mathcal{Z}}(t)$ of
$\tilde w(t,\cdot):=\tilde u(t,\cdot)-\tilde u(t,-\cdot)$ over $\tilde I(t):=[-\tilde k(t), \tilde k(t)]$. Note that due to \eqref{assu-gamma},
case (i) never happens to $\tilde w$.

Now, at each time $t_1>T_1$ such that  $x(t_1)=x_0$, we have
$$
\tilde w(t_1, 0) =\tilde w_x (t_1, 0) =0.
$$
In other words, $0$ is a degenerate zero of $\tilde w(t_1, \cdot)$ in the interior of
$
\tilde I(t_1)$.
 However, the properties of $\tilde{\mathcal{Z}}(t)$ imply that
only finitely many such $t_1$ can exist. This contradiction finishes the proof.
\qed

\section{Transition Speed of the Free Boundary in the Bistable Case}\label{sec:bi}

In this section, we prove Theorem 1.2 for the bistable case.

\begin{thm}\label{thm:propagation-speed-bi}
Assume Theorem {\rm C (iii)} holds and $f$ satisfies \eqref{1+alpha}. Then
$$
-g(t),\ h(t)=\lambda_0\ln t  + O (1) \quad \mbox{with } \lambda_0=[-f'(0)]^{-1/2}.
$$
\end{thm}

\begin{proof}
We only consider the estimate for $h(t)$, as that for $-g(t)$ is similar.
We will prove the estimate by constructing suitable upper and lower solutions. Let $V(x)$ be as in Theorem C.
Since $V(x)$ is even, and $V'(x)<0$ for $x>0$, and $V(+\infty)=0$, for each $m>0$ and $t>m/V(0)$, there exists a unique
$\xi(t)=\xi_m(t)\in (0, +\infty)$ such that
\[
V(\xi(t))=m/t.
\]
Moreover, $\xi(t)$ is a $C^2$ function satisfying $ \xi'(t)>0$ and $\xi(+\infty)=+\infty$.

For clarity we divide the proof below into three steps.

\noindent
{\bf Step 1:} {\it We show the following asymptotic behavior of $\xi(t)$: As $t\to\infty$,}
\begin{equation}
\label{xi}
\xi'(t)=\frac{\lambda_0}{t}[1+o(1)],\;\;
\xi(t)=\lambda_0\ln t+O(1).
\end{equation}

 Multiplying $V''+f(V)=0$ by $V'$ and integrating in $(-\infty, x)$ we obtain
 \[
 \mbox{ $V'^{2}(x)=F(V(x))$, where
$F(u):=-2\int^{u}_{0}f(s)ds$.}
\]
 Therefore
\begin{align*}
V'(\xi(t))
&=-\sqrt{F\left(\frac{m}{t}\right)}=-\sqrt{\frac{1}{2}F''(0)\left(\frac{m}{t}\right)^2+o\left[\left(\frac{m}{t}\right)^2\right]}\\
&=-[-f'(0)]^{1/2}[1+o(1)]\frac{m}{t} \mbox{ as } t\to\infty.
\end{align*}
Differentiating $V(\xi(t))=\frac{m}{t}$ with respect to $t$ we deduce
\begin{align*}
\xi'(t)=-\frac{m}{t^{2}}[V'(\xi(t))]^{-1}
= \frac{1}{ t}[1+o(1)] [-f'(0)]^{-1/2}
=\frac{\lambda_0}{t}[1+o(1)] \mbox{ as } t\to \infty.
\end{align*}
This proves the first part of \eqref{xi}. To show the second part, we use $V'(x)^2=F(V(x))$ again to obtain
\begin{align*}
 x&=\int^{V(0)}_{V(x)} \frac{ ds}{\sqrt{F(s)}}=-\lambda_0 \ln \frac{V(x)}{V(0)}
 + \int^{V(0)}_{V(x)}\Big(\frac{1}{\sqrt{F(s)}}-\frac{\lambda_0}{ s} \Big){ds}\\
 &=-\lambda_0\ln\frac{V(x)}{V(0)}+A_0+o(1) \mbox{ as } x\to+\infty,
 \end{align*}
where
\[
A_0=\int_0^{V(0)}\Big(\frac{1}{\sqrt{F(s)}}-\frac{\lambda_0}{ s} \Big){ds} \mbox{ is finite due to \eqref{1+alpha}}.
\]
It follows that
\[
V(x)=A e^{-\lambda_0^{-1}x}[1+o(1)] \mbox{ as $x\to+\infty$},
\mbox{ with } A=V(0)e^{\lambda_0^{-1}A_0}.
\]
Therefore
\[
\xi(t)=
 V^{-1}\left(\frac{m}{t}\right)=\lambda_0 \ln t-\lambda_0\ln \frac{m}{A}+o(1) \mbox{ as $t\rightarrow +\infty$,}
 \]
which implies the second part of \eqref{xi}.
\smallskip

\noindent
{\bf Step 2:} {\it We obtain a lower bound for $h(t)$.}

By Theorem \ref{conv}, there exists $x_0\in [-h_0,h_0]$ such that
\[
\lim_{t\to\infty} u(t,x)=V(x+x_0) \mbox{ locally uniformly  in } x\in\R^1.
\]
We now define $(\underline{V},\underline{h})$ by
$$
\begin{cases}
\underline{h}(t)=\xi(t)-x_0-1,& t>m/V(0),\\
\underline{V}(t,x)=V(x+x_0 +1)-\frac{m}{t}, & x\in (0,\underline{h}(t)),t>m/V(0).
\end{cases}
$$
Clearly $\underline V(t, \underline h(t))=V(\xi(t))-\frac{m}{t}=0$.
We will show that by choosing $m>0$, $M_0>0$ and $T>0$ properly, $(\underline{V},\underline{h})$ satisfies
\[
\underline h(t)\leq h(t+T) \mbox{ and }
\underline V(t,x)\leq u(t+T, x) \mbox{ for } M_0\leq x\leq \underline h(t) \mbox{ and all large } t.
\]

Since $f$ is a bistable nonlinearity, there exists $\rho\in (0,\theta)$ such that
$f(s)<0$ and $f'(s)<\frac{1}{2}f'(0)$ for $s\in (0, \rho)$. Choose $T_1>0$
such that for $t>T_1$, $\underline{h}(t)> M_0:= V^{-1} (\rho) -x_0 -1$ and hence
$V(\underline{h}(t) +x_0 +1) <\rho$.
Then for $x\in [M_0,\underline h(t)]$ and $t>T_1$, we have
\[
\mbox{ $\underline{h}(t)> M_0$, $V(x+x_0+1) \leq \rho$,
$f(V(x+x_0+1))-f(\underline V(t,x))<\frac{1}{2}f'(0)\frac{m}{t}$.}
\]
We next show that for sufficiently large $t$ and $x\in [M_{0},\underline{h}(t))$,
\begin{equation}\label{sub1}
\underline{V}_{t}-\underline{V}_{xx}-f(\underline{V})\leq 0,
\end{equation}
\begin{equation}\label{sub2}
\underline{h}'(t)\leq -\mu\underline{V}_{x}(t,\underline{h}(t)).
\end{equation}
Indeed, for $t>T_2:= \max\{-\frac{2}{f'(0)}, T_1\}$ and $x\in[M_{0},\underline{h}(t))$, we have, with $V=V(x+x_0+1)$,
$$
\underline{V}_{t}-\underline{V}_{xx}-f(\underline{V})=\frac{m}{t^2}+f(V)-f(\underline V)
< \frac{m}{t^{2}}+\frac{1}{2}f'(0)\frac{m}{t}  <0.
$$
This proves \eqref{sub1}.

We now prove \eqref{sub2}.
By our estimates in Step 1, clearly
\begin{align*}
\underline{V}_{x}(t,\underline{h}(t))&=V'(\xi(t))\\
&=-[-f'(0)]^{1/2}[1+o(1)]\frac{m}{t} \mbox{ as } t\to\infty,
\end{align*}
and
\[
\underline{h}'(t)=\xi'(t)= \frac{\lambda_0}{ t}[1+o(1)] \mbox{ as } t\to \infty.
\]
 Thus \eqref{sub2} holds for all large $t$, say $t>T_3 \geq T_2$, provided that $m$ is chosen such that
 \[
 \mu[-f'(0)]^{1/2} m>\lambda_0, \mbox{ i.e., }
 m>\lambda_0^2/\mu.
 \]

We fix $m$ as above, and now compare $\underline V(t, M_0)$ with $u(t, M_0)$. Clearly
\[
\lim_{t\to\infty} u(t, M_0)=V(M_0+x_0)>V(M_0+x_0+1)=\lim_{t\to\infty} \underline V(t, M_0).
\]
Therefore we can find a time $T_4>T_3$ such that
\begin{equation}
\label{sub3}
\underline{V}(t+T_4, M_0) < u(s+T_4, M_0)\ \ \mbox{ for all } t, s >0.
\end{equation}
Moreover, since $\underline V(T_4,x)=V(x+x_0+1)-\frac{m}{T_4}<V(x+x_0)-\frac{m}{T_4}$, and
since
\[
\lim_{t\to\infty} h(t)=+\infty,\; \lim_{t\to\infty} u(t,x)= V(x+x_0) \mbox{ locally uniformly  for $x\in \R^1$},
\]
 there exists $T_5>T_4$ such that
\begin{equation}\label{sub4}
\underline{h}(T_4) <  h(T_5), \quad \underline{V}(T_4,x)\leq u(T_5,x) \ \mbox{ for } x\in [M_0, \underline{h}(T_4)].
\end{equation}
Combining \eqref{sub1}, \eqref{sub2}, \eqref{sub3} and \eqref{sub4}, we see that upon using the comparison principle,
for $x\in [M_0, \underline{h}(t+T_4)]$ and $t>0$, we have
\[
h(t+T_5)\geq \underline{h}(t+T_4),\; u(t+T_5, x)\geq \underline V(t+T_4, x).
\]
In view of \eqref{xi}, the above inequality for $h(t+T_5)$ clearly implies that
\[
h(t)\geq \lambda_0\ln t-M_1 \mbox{ for all large $t$ and some $M_1>0$.}
\]

\smallskip

\noindent
{\bf Step 3:} {\it We obtain an upper bound for $h(t)$.}

To complete the proof, it remains to show an estimate of the form
\begin{equation}
\label{h-ubd}
h(t)\leq \lambda_0\ln t+M_2 \mbox{ for all large $t$ and some $M_2>0$.}
\end{equation}
We will accomplish this by constructing suitable upper solutions.

Define $(\bar{V},\bar{h})$ by
$$
\begin{cases}
\bar h(t)=\xi_{m_1}(t)+\frac{2}{3}\lambda_0-x_0+1,& t>m_1/V(0),\\
\bar{V}(t,x)=V(x+x_0-1)+\frac{m_1}{t}, & x\in[0,\bar{h}(t)-\frac{2}{3}\lambda_0],\ t>m_1/V(0),\\
\bar{V}(t,x)=\frac{3}{\lambda_0}\frac{m_1}{t}(\bar{h}(t)-x),& x\in[\bar{h}(t)-\frac{2}{3}\lambda_0,\bar{h}(t)], \ t>m_1/V(0),\\
\end{cases}$$
where $m_1:=\frac{\lambda_0^2}{4\mu}$.
We are going to show that there exists $M^0>0, T_7> T_6>0$ such that
\begin{equation}\label{sup1}
\bar{V}_{t}-\bar{V}_{xx}-f(\bar{V})\geq 0\quad  \mbox{ for } x\in \Big[ M^0,\bar{h}(t)-\frac{2}{3}\lambda_0 \Big),\
t>T_7,
\end{equation}
\begin{equation}\label{sup2}
\bar{V}_{t}-\bar{V}_{xx}-f(\bar{V})\geq 0\quad  \mbox{ for }  x\in \Big( \bar{h}(t)-\frac{2}{3}\lambda_0,
\bar{h}(t) \Big),\ t>T_7,
\end{equation}
\begin{equation}\label{sup3}
\lim_{x\rightarrow [\bar{h}(t)-\frac{2}{3}\lambda_0]^{+}} \bar{V}_{x}(t,x)
\leq \lim_{x\rightarrow [\bar{h}(t)-\frac{2}{3}\lambda_0]^{-}} \bar{V}_{x}(t,x)\quad \mbox{ for } t>T_7,
\end{equation}
\begin{equation}\label{sup4}
\bar{h}'(t)\geq -\mu \bar{V}_{x}(t,\bar{h}(t)) \mbox{ and } \bar V(t,\bar h(t))=0 \quad \mbox{ for } t>T_7,
\end{equation}
\begin{equation}
\label{sup5}
\bar V(t+T_7, M^0)\geq u(t+T_6, M^0) \mbox{ for } t\geq 0,
\end{equation}
\begin{equation}
\label{sup6}
 \bar h(T_7)>h(T_6) \mbox{ and } \bar V(T_7,x)\geq u(T_6,x) \mbox{ for } x\in [M^0, h(T_6)].
\end{equation}
If \eqref{sup1} through to \eqref{sup6} hold, then we can apply the comparison principle to deduce that
\[
\bar h(t+T_7)\geq h(t+T_6),\; \bar V(t+T_7,x)\geq u(t+T_6,x) \mbox{ for $x\in [M^0, h(t+T_6)]$ and $t>0$.}
\]
Applying \eqref{xi} to $\xi_{m_1}(t)$ and using the definition of $\bar h$, we immediately obtain \eqref{h-ubd}, as wanted.

Next we prove \eqref{sup1}-\eqref{sup6} one by one, starting  with \eqref{sup1}.
We choose $M^0 >0$ large so that for all large $t$ and $x\in [M^0, \bar h(t)-\frac23 \lambda_0]$, $V(x+x_0-1)+\frac{m_1}{t}$ is small enough so that, with $V$ standing for $V(x+x_0-1)$,
\[
f(V)-f\left(V+\frac{m_1}{t}\right)>-\frac12 f'(0)\frac{m_1}{t}.
\]
Then for such $t$ and $x$,
\begin{align*}
\bar V_t-\bar V_{xx}-f(\bar V)&=-\frac{m_1}{t^2}+f(V)-f\left(V+\frac{m_1}{t}\right)\\
&\geq -\frac{m_1}{t^2}-\frac12 f'(0)\frac{m_1}{t}>0.
\end{align*}
This proves \eqref{sup1}.

To prove \eqref{sup2}, we note that, for $x\in [\bar h(t)-\frac 23 \lambda_0, \bar h(t)]$ and large $t$,
\[
\bar V_{xx}=0,\; f(\bar V(t,x))\leq 0,
\]
and by direct calculation and \eqref{xi},
\begin{align*}
\bar V_t&=-\frac{3}{\lambda_0}\frac{m_1}{t^2}(\bar h(t)-x)+\frac{3}{\lambda_0}\frac{m_1}{t}\bar h'(t)\\
&\geq -\frac{3}{\lambda_0}\frac{m_1}{t^2}\frac 23 \lambda_0+\frac{3}{\lambda_0}\frac{m_1}{t}\frac{\lambda_0}{t}[1+o(1)]\\
&=-2\frac{m_1}{t^2}+3\frac{m_1}{t^2}[1+o(1)]>0.
\end{align*}
Hence \eqref{sup2} holds.

Clearly \eqref{sup3} follows from
\[
\lim_{x\to [\bar h(t)-\frac23 \lambda_0]^+}\bar V_x(t, x)=-\frac{3}{\lambda_0}\frac{m_1}{t}
\]
and
\[
\lim_{x\to [\bar h(t)-\frac23 \lambda_0]^-}\bar V_x(t, x)=V'(\xi_{m_1}(t))=-\frac{1}{\lambda_0}\frac{m_1}{t}[1+o(1)]
\mbox{ as } t\to\infty.
\]

By definition, $\bar V(t,\bar h(t))=0$.
By direct calculation and the choice of $m_1$ we have
\[
-\mu\bar{V}_{x}(t,\bar{h}(t))=\frac{3\mu}{\lambda_0}\frac{m_1}{t}=\frac 34 \frac{\lambda_0}{t}.
\]
By \eqref{xi},
\[
\bar{h}'(t)=\xi_{m_1}'(t)= \frac{\lambda_0}{t}[1+o(1)] \mbox{ as } t\to \infty.
\]
Therefore  \eqref{sup4} holds for all large $t$.

Finally the inequalities in \eqref{sup5} and \eqref{sup6} are easy consequences of the facts that
\[
\lim_{t\to\infty}\bar h(t)=+\infty,\; \lim_{t\to\infty}\bar V(t,x)=V(x+x_0-1)>V(x+x_0)=\lim_{t\to\infty} u(t,x)
\]
uniformly for $x$ in any compact subset of $[M^0,+\infty)$.

The proof of the theorem is now complete.
\end{proof}

\section{Transition Speed of the Free Boundary in the Combustion Case}\label{sec:com}

In this section we prove Theorem 1.2 for the combustion case. So throughout this section, we always assume that $f$ is of combustion type. Our proof is rather involved.
 For clarity, we divide our analysis into several subsections. 

\subsection{A key lemma}

To stress the dependence of the unique solution $(u(t,x), h(t), g(t))$ of \eqref{p} on the initial function $u_0$, we will write
\[
u(t,x)=u(t,x;u_0),\; h(t)=h(t;u_0),\; g(t)=g(t; u_0).
\]
For convenience, we always think of $u_0(x)$ as defined for all $x\in\R^1$, with value zero outside its supporting set.
Moreover, in this subsection, we drop the assumption that the supporting set of $u_0$ is symmetric about $x=0$, which was assumed in \eqref{p} for convenience.

\begin{lem}\label{lem:finite dif}
Given any two compactly supported initial functions $\{\phi_i\}_{i=1,2}$ of \eqref{p} such that
$$
\lim_{t\to\infty}u(t,x;\phi_i) = \theta \ \mbox{locally uniformly in } \R^1,
$$
 there exists a constant $M>0$ such that
$$
|h(t;\phi_1) -h(t;\phi_2)|\leq M,\quad |g(t;\phi_1) -g(t;\phi_2)|\leq M\quad \mbox{for } t>0.
$$
\end{lem}

\begin{proof}
By Lemma 2.8 of \cite{DLou}, $\frac{1}{2}[g(t;\phi_i)+h(t;\phi_i)]$ is contained in the supporting set of $\phi_i$ for all $t>0$.
Therefore it suffices to show $|h(t;\phi_1) -h(t;\phi_2)|\leq M$ for all $t>0$.
We will only prove
\begin{equation}
\label{4.1}
h(t;\phi_1)-h(t;\phi_2)\leq M,
\end{equation}
 since $h(t;\phi_2)-h(t;\phi_1)\leq M$ can be proved in the same way.

Because both  $\phi_1$ and
$\phi_2$ are compactly supported, by replacing $\phi_1$ with $\phi_1(\cdot+M_1)$ for some large positive constant $M_1$, we may assume that
\begin{equation}
\label{M1}
h(0;\phi_1) < g(0;\phi_2) -1.
\end{equation}
 So the support of $\phi_1$ lies to the left of the support of $\phi_2$ with some positive distance.

 As $t$ is increased to $+\infty$, by assumption $h(t;\phi_1)$ increases to $+\infty$, and $g(t;\phi_2)$ decreases to $-\infty$.
Therefore there exists a unique time $T_1>0$ such that
\[
h(T_1; \phi_1) = g(T_1; \phi_2).
\]
We now consider $h(t; \phi_1)-h(t;\phi_2)$, which is negative when $t=0$ due to \eqref{M1}. If $h(t; \phi_1)-h(t;\phi_2)<0$ for all $t\geq 0$,
then \eqref{4.1} holds for any $M>0$,  which is what we wanted.

In the following, we consider the remaining case, where
$$
T_2 := \inf \{t>0: h(t;\phi_1) = h(t;\phi_2)\}
$$
is  finite and positive.
Let us also define
$$
T_3 := \inf \{t>0: g(t;\phi_1) = g(t;\phi_2)\},
$$
with the convention that $T_3=+\infty$ if $g(t;\phi_1) - g(t;\phi_2)<0$ for all $t>0$. Note that by
\eqref{M1}, $T_3>0$ whenever it is finite. It is easily seen that $T_2, T_3 >T_1$.

For $t>T_1$, we now define
\[
k_l(t):=\max\big\{g(t;\phi_1), g(t; \phi_2)\big\},\; k_r(t):=\min\big\{h(t;\phi_1), h(t; \phi_2)\big\},
\]
and
\[ w(t,x):=u(t,x; \phi_1)-u(t,x;\phi_2),\; x\in [k_l(t), k_r(t)].
\]
Then similar to the situation in Lemma \ref{lem2.3}, $w$ satisfies
\[
w_t-w_{xx}=c(t,x)w \mbox{ for } x\in [k_l(t), k_r(t)],\; t>T_1,
\]
where $c\in L^\infty$, and
\begin{equation}\label{4.2}
w(t, k_l(t))>0>w(t, k_r(t)) \mbox{ for } T_1<t<T^1:=\min\{T_2, T_3\}.
\end{equation}

Note that $k_l(T_1)=k_r(T_1)=g(T_1;\phi_2)=h(T_1;\phi_1)$, and for $t>T_1$ but very close to $T_1$, $w(t,\cdot)$ has exactly one zero in $[k_l(t),k_r(t)]$, which is nondegenerate.
Indeed, for such $t$ and $x\in [k_l(t), k_r(t)]$, by continuity and  the Hopf boundary lemma,
\[
u_x(t,x;\phi_1)<\frac12 u_x(t,h(t;\phi_1);\phi_1)<0,
\]
and
\[
u_x(t,x;\phi_2)>\frac12 u_x(t,g(t;\phi_2);\phi_2)>0.
\]
Hence $w_x(t,x)<0$.

Let $x=x(t)$ denote this unique zero for $t>T_1$ but close to $T_1$.  In view of this fact and \eqref{4.2},
we can apply Lemma 2.2 to conclude that for every $t\in (T_1, T^1)$, $w(t,x)$ has at most one zero in $(k_l(t), k_r(t))$. On the other hand, \eqref{4.2} implies that for every such $t$, $w(t,x)$ has at least one zero.
Therefore there is a unique zero, and by Lemma 2.2, it must be nondegenerate. Thus the unique zero $x(t)$
is defined for all $t\in (T_1, T^1)$, and due to its nondegeneracy, $x(t)$ is a $C^1$ function.

We now consider the limit of $x(t)$ as $t$ increases to $T^1$. If it does not exist, then
 as in  the proof of Lemma 2.4 we deduce that 
 \begin{equation}
\label{4.3}
 \mbox{$w(T^1,\cdot)\equiv 0$ in $[k_l(T^1), k_r(T^1)]$;}
  \end{equation}
  if  $\lim_{t\nearrow T^1}x(t)=x(T^1)$ exists, we can also argue as in the proof of Lemma 2.4
to see that $x(T^1)\in\{k_l(T^1), k_r(T^1)\}$, and $w(T^1,x)$ does not change sign in $(k_l(T^1), k_r(T^1))$.
   
   In the former case, by the uniqueness of the solution to the free boundary problem, we must have
\[
\big(u(t,x;\phi_1), g(t; \phi_1), h(t;\phi_1)\big)\equiv \big(u(t,x;\phi_2),g(t;\phi_2), h(t;\phi_2)\big)\;\;\;\;\;\;
\mbox{ for } t>T^1,
\]
which clearly implies \eqref{4.1}. 

In the latter case we show that a contradiction arises, and so this case cannot occur.
Indeed, we have
\[\mbox{
$[k_l(T^1), k_r(T^1)]=[g(T^1;\phi_2), h(T^1;\phi_2)]$ when $x(T^1)=k_r(T^1)$,}
\]
 and
\[\mbox{
$[k_l(T^1), k_r(T^1)]=[g(T^1;\phi_1), h(T^1;\phi_1)]$ when $x(T^1)=k_l(T^1)$.}
\]
Therefore, when $x(T^1)=k_r(T^1)$, we have
\[
u(T^1,x;\phi_1)>u(T^1,x;\phi_2) \mbox{ for } x\in (g(T^1;\phi_2), h(T^1,\phi_2)).
\]
By the comparison principle (see Lemma 2.1 of \cite{DLou}) and the strong maximum principle we deduce that, for  $t>T^1$,
$g(t; \phi_1)<g(t;\phi_2)$, $h(t;\phi_1)>h(t;\phi_2)$ and
\[
u(t,x;\phi_1)>u(t,x;\phi_2) \mbox{ for $g(t;\phi_2)\leq x\leq h(t;\phi_2)$}.
\]
Hence for fixed $t_0>T^1$,
there exists $\epsilon>0$ such that
\[
u(t_0,x;\phi_1)\geq (1+\epsilon)u(t_0,x;\phi_2) \mbox{ for } x\in [g(t_0;\phi_2), h(t_0;\phi_2)].
\]
This implies, by the sharp transition result of \cite{DLou} and the comparison principle, 
\[
\lim_{t\to\infty} u(t,x;\phi_1)=1.
\]
On the other hand, by the choice of $\phi_1$, we have
\[
u(t,x;\phi_1)\to \theta \mbox{ as } t\to\infty.
\]
When $x(T^1)=k_l(T^1)$, we can derive a contradiction similarly. 
The proof of the lemma is complete.
\end{proof}

By Lemma \ref{lem:finite dif}, to prove Theorem 1.2 for the combustion case, it suffices to consider a
special initial function. This will be crucial to our analysis. More precisely, suppose that transition happens with the initial function $u_0\in\mathscr{X}(h_0)$.  We choose a function $\tilde u_0\in\mathscr{X}(h_0)$ with the properties
$$
\tilde u_0(x)\geq u_0(x) \mbox{ in } [-h_0,h_0],\;\; \tilde u_0(x)=\tilde u_0(-x) \mbox{ and } \tilde u_0'(x)<0 \ \ \mbox{for}\ \ 0<x<h_0.
$$
Then there exists $\sigma^*\in (0,+\infty)$ such that
transition happens for the solution of \eqref{p} with initial function $\phi:=\sigma^* \tilde u_0$.
By uniqueness and a simple reflection-comparison argument, such a solution satisfies
\begin{equation}\label{special-u}
u(t,x)=u(t,-x),\quad u_x(t,x)<0\ \ \mbox{for}\ \ 0<x<h(t), \, t\geq 0.
\end{equation}
Hence $g(t)=-h(t)$, and we only need to show that 
\[
h(t)=2\xi_0\sqrt{t}[1+o(1)] \mbox{ as } t\to\infty.
\]

\subsection{Estimate of $h(t)$  under an extra condition}

As explained above, by Lemma 4.1 we only need to consider the transition case
with a special initial function such that the solution $u(t,x)$ satisfies \eqref{special-u}.

We first observe that $u(t, 0)>\theta$ for all $t\geq 0$. Otherwise there exists $t_0\geq 0$ such that $u(t_0,0)
\leq \theta$ and hence $u(t_0,x)<\theta$ for $x\in [-h(t_0), h(t_0)]\setminus\{0\}$. By the strong maximum principle we easily deduce $u(t,x)<\theta$ for $t>t_0$ and $x\in [-h(t),h(t)]$, which implies $u(t,x)\to 0$ as $t\to\infty$ (see \cite{DLou}), contradicting the assumption that transition happens.

This observation and \eqref{special-u} indicate that for each $t\geq 0$, there is a unique $\theta(t)\in (0, h(t))$ such that
\[
u(t,\theta(t))=\theta.
\]
We will prove that
\begin{equation}\label{assump}
\lim_{t\to\infty}\theta(t)/h(t)=0.
\end{equation}

Assuming \eqref{assump}, we now prove the required estimate for $h(t)$, namely

\begin{prop}\label{prop:propagation-speed-com}
Let $u(t,x)$, $h(t)$ and $\theta(t)$ be as above. Suppose that \eqref{assump} holds.
Then
\begin{equation}\label{4.7}
h(t)= 2\xi_0 \sqrt{t} \ [1+o(1)] \quad \mbox{as } t\to \infty,
\end{equation}
where $\xi_0$ is  defined by \eqref{def-xi_0}.
\end{prop}

\begin{proof}

We will prove \eqref{4.7} by some comparison arguments involving the functions $\Phi(t,x)$ and $\rho(t)$ given by
\[
\Phi (t,x) :=  \frac{\theta}{E(\xi_0)} \Big[E(\xi_0) - E\left(\frac{x}{2\sqrt{t}}\right) \Big],\;\;  \rho(t):= 2\xi_0 \sqrt{t},
\]
where $E(x):= \frac{2}{\sqrt{\pi}} \int_0^x e^{-t^2} dt$ and $\xi_0$ is given by \eqref{def-xi_0}.

It is easily seen that
\[
\Phi(t,0)=\theta,\; \Phi(t,\rho(t))=0 \mbox{ and } \Phi(t,x)>0 \mbox{ for $t>0$ and } x\in [0, \rho(t)).
\]
In fact, a direct calculation confirms the well-known fact  that $(\eta, r)=(\Phi(t,x), \rho(t))$ satisfies
\begin{equation}\label{1d Stefan}
\left\{
\begin{array}{ll}
\eta_t -\eta_{xx} =0, & 0<x<r(t),\ t>0,\\
\eta(t,0) =\theta,\ \eta(t,r(t))=0, & t> 0,\\
r'(t)= -\mu \eta_x (t,r(t)), & t> 0.
\end{array}
\right.
\end{equation}

For $\epsilon\in (0,1)$ and $T>0$ to be determined, we define
\[
\tilde h(t)=(1-\epsilon)^{-1}\rho(t+T),\;\tilde \eta(t,x)=\Phi(t+T, x-\epsilon \tilde h(t)) \mbox{ for $t>0$ and $x\in [\epsilon\tilde h(t), \tilde h(t)]$}.
\]
It is clear that 
\[
\tilde \eta(t,\epsilon \tilde h(t))=\theta,\; \tilde \eta (t, \tilde h(t))=0,
\]
and
\[
\tilde h'(t)=(1-\epsilon)^{-1}\rho'(t+T)>\rho'(t+T)=-\mu \Phi_x(t+T, \rho(t+T))=-\mu \tilde \eta_x(t,\tilde h(t)).
\]
By the definition, we find
\[
\tilde \eta(t,x)=\frac{\theta}{E(\xi_0)}\left[E(\xi_0)-E\left(\frac{x}{2\sqrt{t+T}}-\frac{\epsilon\xi_0}{1-\epsilon}\right)\right],
\]
from which it is easily calculated that
\[
\tilde \eta_t-\tilde \eta_{xx}=0.
\]

Next we determine $T$ for any fixed $\epsilon\in (0,1)$. 
By \eqref{assump},  there exists
$\tau_\epsilon >0$ such that
\begin{equation}\label{x< epsilon h}
\theta(t) < \epsilon h(t)\ \ \mbox{for}\ t>\tau_\epsilon.
\end{equation}
We choose $T=T_\epsilon$ such that
\[
h(\tau_\epsilon)=\frac{\epsilon}{1-\epsilon} \rho(T)=\epsilon\tilde h(0).
\]
Then $h(\tau_\epsilon)<\tilde h(0)$ and by continuity we can find $\delta_0>0$ such that
\[
h(\tau_\epsilon+t)< \tilde h(t) \mbox{ for } t\in [0, \delta_0).
\]
We claim that 
\[
h(\tau_\epsilon+t)<\tilde h(t) \mbox{ for all } t\geq 0.
\]
Otherwise we can find $t_0\geq \delta_0$ such that
\[
h(\tau_\epsilon+t)<\tilde h(t) \mbox{ for } t\in [0, t_0),\; h(\tau_\epsilon+t_0)=\tilde h(t_0).
\]
It follows that
\[
h'(\tau_\epsilon+t_0)\geq \tilde h'(t_0).
\]

On the other hand, by \eqref{x< epsilon h},
\[
\theta(\tau_\epsilon+t)<\epsilon h(\tau_\epsilon+t)<\epsilon\tilde h(t) \mbox{ for } t\in [0,t_0),
\]
which implies that $u(\tau_\epsilon+t, \epsilon\tilde h(t))<\theta$ for $t\in [0, t_0)$.
This allows us to compare $u(\tau_\epsilon+t, x)$ with $\tilde \eta(t,x)$ by the comparison principle over the region
$\Omega:=\{(t,x): \epsilon\tilde h(t)< x< h(\tau_\epsilon+t), 0< t\leq t_0\}$ to conclude that
\[
u(\tau_\epsilon+t,x)< \tilde \eta(t,x) \mbox{ in } \Omega.
\]
 By the Hopf boundary lemma we further obtain
\[
u_x(\tau_\epsilon+t_0, h(\tau_\epsilon+t_0))>\tilde \eta_x(t_0, \tilde h(t_0)).
\]
It follows that
\[
h'(\tau_\epsilon+t_0)=-\mu u_x(\tau_\epsilon+t_0, h(\tau_\epsilon+t_0))<-\mu\tilde \eta_x(t_0, \tilde h(t_0))<\tilde h'(t_0).
\]
This contradiction proves our claim.

Thus we have
$$
h(t+\tau_\epsilon) < \tilde{h}(t)=\frac{2\xi_0}{1-\epsilon} \sqrt{t + T_\epsilon}
\quad \mbox{for } t>0.
$$
It follows that
$$
\limsup\limits_{t\to \infty} \frac{h(t)-2\xi_0 \sqrt{t}}{\sqrt{t}} \leq
\lim_{t\to \infty} \frac{\tilde h(t-\tau_\epsilon)-2\xi_0 \sqrt{t}}{\sqrt{t}}=2\xi_0\frac{\epsilon}{1-\epsilon}.
$$
Since $\epsilon\in (0, 1)$ is arbitrary, we obtain
\begin{equation}
\label{h-upper}
\limsup\limits_{t\to \infty} \frac{h(t)-2\xi_0 \sqrt{t}}{\sqrt{t}} \leq 0.
\end{equation}

Next we estimate $h(t)$ from below.
Since  $u(t,0)\to \theta$ as $t\to \infty$, we can find $t_1>0$ such that $u(t_1+t,x)<1$ for all $t\geq 0$ and $x\in [0, h(t_1+t)]$. It follows that $f(u(t_1+t, x))\geq 0$ for such $t$ and $x$. We claim that
$$
\rho(t) < h(t_1+t)\quad
 \mbox{for}\ t>0.
$$
If this is not true, then we can find $\tilde t_0>0$ such that 
\[
\rho(t) < h(t_1+t)\quad
 \mbox{for}\ t\in (0, \tilde t_0),\; \rho(\tilde t_0)=h(t_1+\tilde t_0).
\]
Then we can apply the comparison principle over the region
$\{(t,x): 0<x<\rho(t), 0<t\leq \tilde t_0\}$ to deduce a contradiction in the same way
as we did above over the region $\Omega$. This proves our claim, and so
\[
h(t) > \rho(t-t_1) =2\xi_0 \sqrt{t} + O \Big( \frac{1}{\sqrt{t}}\Big)\quad \mbox{as}\ t\to \infty.
\]
The required estimate \eqref{4.7} is a direct consequence of this fact and \eqref{h-upper}.
\end{proof}

\subsection{Proof of \eqref{assump}.} To complete the proof of Theorem 1.2, it remains to prove \eqref{assump}.
We will do this in several steps.

\subsubsection{\bf Analysis of an ODE problem}

For any $b\in (0,(1-\theta)/2)$, we consider the initial value problem
$$
v''+f(v)=0, \quad v(0)=\theta+b,\ \ v'(0)=0.
$$
Let  $V_b(x)$ denote its unique solution.

\begin{lem}\label{lem:lb-b}
There exist $0<l(b)<L(b)<+\infty$ such that
\begin{itemize}
\item[(i)] $ V_b(l(b))=\theta\quad \mbox{and}\quad V_b(L(b))=0,
$
\item[(ii)]
$
V_b(x)=V_b(-x) \ \mbox{and}\ V'_b(x)<0 \ \mbox{for } x\in (0,L(b)],$
\item[(iii)] $l(b)\to \infty$ if and only if $b\to 0$.
\end{itemize}
\end{lem}

\begin{proof} The conclusions follow directly from a simple phase plane analysis. The details are omitted.
\end{proof}

\begin{lem}\label{lemm-key}
$
\lim_{b\to 0}  l(b)/L(b)=\lim_{b\to 0} V_b'(l(b)) =0.
$
\end{lem}

\begin{proof} Since $V_b(x)$ is a linear function over $[l(b), L(b)]$, we have
\[
\frac{\theta}{L(b)-l(b)}=-V_b'(l(b)) \mbox{ and  hence } 0<\frac{l(b)}{L(b)}<-V_b'(l(b))\frac{l(b)}{\theta}.
\]
Because of $V''_b +f(V_b) =0$ and $V'_b (0)=0$, it follows that
$$
V'_b (l(b)) = -\sqrt{G(b)} \ \quad \mbox{with} \ \quad G(u):= 2\int_0^u f(s+\theta) ds,
$$
and
$$
l(b)= \int_0^b \frac{1}{\sqrt{G(b)- G(s)}} ds.
$$
Therefore
\begin{equation}\label{lb-b}
l(b)V'_b(l(b)) = -\int_0^b \frac{\sqrt{G(b)}}{\sqrt{G(b) -G(s)}} ds
= -b \int_0^1 \frac{\sqrt{G(b)}} {\sqrt{ G(b) - G(br)} } dr.
\end{equation}

Since $f(u)$ is nondecreasing in $u\in (\theta, \theta+\delta)$, for any $0<r<1$ and $0<b<\delta$ we have
\[
G(br) = 2\int_0^{br} f(s+\theta) ds = 2r\int_0^b f(rt+\theta) dt \leq 2r\int_0^b f(t+\theta)dt =rG(b).
\]
Substituting this into \eqref{lb-b} we obtain, for $b\in (0,\delta)$,
$$
0> l(b) V'_b(l(b)) \geq -b \int_0^1 \frac{1}{\sqrt{1-r}} dr = -2b.
$$
It follows that 
\[
 0<\frac{l(b)}{L(b)}<-V_b'(l(b))\frac{l(b)}{\theta}\leq \frac{2b}{\theta }.
\]
 Since $l(b)\to +\infty$ as $b\to 0$,
the lemma is proved.
\end{proof}

\subsubsection{\bf Sign-changing patterns of the function ``$x\mapsto u(t,x)-V_b(x)$''.}

In this step, we classify the sign-changing patterns of the function 
\[
w_b(t,x):=u(t,x)-V_b(x)
\]
 for any fixed $t\geq 1$ and small $b>0$. This will be done by making use of the comparison principle and the zero number argument.

Let us recall that $u(t,x)$ satisfies \eqref{special-u}, and $u(t,x)\to \theta$ in $C^1_{loc}(\R^1)$ as $t\to\infty$.
Moreover, $u(t,0)>\theta$ for all $t\geq 0$.

\begin{lem}\label{t=1}
There exists $\delta_0>0$ small such that for each $b\in (0,\delta_0)$,
\begin{itemize}
\item[(i)]  $w_b(1,0)>0>w_b(1, h(1))$,
\item[(ii)] $w_b(1,x)$ has a unique zero in
$ [0, h(1)]$, and the zero is nondegenerate.
\end{itemize}
\end{lem}
\begin{proof}
Fix $a\in (\theta, u(1,0))$. There exists $x_a\in (0, h(1))$ and $\epsilon_a>0$ such that
\[
u(1,x_a)=a,\; u_x(1,x)\leq -\epsilon_a \mbox{ for } x\in [x_a, h(1)].
\]
Since $V_b''(x)\leq 0$ for  $x\in [0, L(b)]$ and $V''_b(x)\equiv 0$ for $x\in [l(b), L(b)]$, we have
\begin{equation}\label{V'}
0\geq V_b'(x)\geq V'(l(b)) \mbox{ for } x\in [0, l(b)],\; V'_b(x)\equiv V'(l(b)) \mbox{ for } x\in [l(b), L(b)].
\end{equation}
Therefore from Lemmas \ref{lem:lb-b} and \ref{lemm-key} we find that
\[
l(b)\to+\infty \mbox{ and } \|V'_b\|_\infty\to 0 \mbox{ as } b\to 0.
\]
Hence we can find $\delta_0\in (0, a-\theta)$ sufficiently small so that, for $b\in (0,\delta_0)$,
\[
l(b)>h(1),\; 
0\geq V_b'(x)>-\epsilon_a \mbox{ for } x\in [0, L(b)].
\]
It follows that, for such $b$, $\frac{d}{dx} w_b(1,x)<0$ for $x\in [x_a, h(1)]$, and
\[
w_b(1, x_a)=a-V_b(x_a)>a-V_b(0)>0,\; w_b(1, h(1))=-V_b(h(1))<0.
\]
Hence $w_b(1,x)$ has a unique zero in $(x_a, h(1))$, and the zero is nondegenerate.

For $x\in [0, x_a]$, we have
\[
w_b(1,x)>u(1,x_a)-V_b(0)=a-V_b(0)>0.
\]
The proof is complete.
\end{proof}

From now on, we always assume that 
\[
b\in (0,\delta_0) \mbox{ with $\delta_0$ given in Lemma \ref{t=1}.}
\]
To simplify notations, we will write $w(t,x)$ instead of $w_b(t,x)$ when the dependence of $b\in (0, \delta_0)$ is not stressed.

Since $t\mapsto w(t,x)$, $t\mapsto w_x(t,x)$ and $t\mapsto h(t)$ are all continuous and uniformly in $x$, from the conclusions of Lemma \ref{t=1} we see that, there exists
$\epsilon_0>0$ small such that for each fixed $t\in [1-\epsilon_0, 1+\epsilon_0]$, $w(t,x)$ has the same properties, namely
\begin{itemize}
\item[(i)]  $w(t,0)>0>w(t, h(t))$,
\item[(ii)] $w(t,x)$ has a unique zero in
$ [0, h(t)]$, and the zero is nondegenerate.
\end{itemize} 

We now define
\[
T_1:=\sup\{s: w(t,0)>0 \mbox{ for } t\in [1-\epsilon_0, s)\},
\]
\[ 
T_2:=\sup\{s: h(t)<L(b) \mbox{ for } t\in [1-\epsilon_0, s)\}.
\]
Clearly $T_1, T_2\geq 1+\epsilon_0$. Since $h(t)\to+\infty$ and $w(t,0)\to -b<0$ as $t\to+\infty$, $T_1$ and $T_2$ are both finite.

\begin{lem}\label{T1<T2}
Suppose  $T_1<T_2$. Then
\begin{itemize}
\item[(i)] for $t\in [1, T_1)$, $w(t,x)$ has a unique nondegenerate zero $x(t)$ in $(0, h(t))$, with sign-changing pattern $[+0-]$ over $[0, h(t)]$, meaning 
\[\mbox{
$w(t,x)>0$ in $[0, x(t))$, $w(t,x(t))=0$, $w(t,x)<0$ in $(x(t), h(t)]$;}\]
\item[(ii)] $w(T_1,x)$ has sign-changing pattern $[0-]$ over $[0, h(T_1)]$, meaning 
\[
\mbox{$w(T_1, 0)=0$ and $w(T_1,x)<0$ in $(0, h(T_1)]$;}\]
\item[(iii)] for $t\in (T_1, T_2)$, $w(t,x)$ has sign-changing pattern $[-]$ over $[0, h(t)]$, meaning 
\[
\mbox{ $w(t,x)<0$ in $[0, h(t)]$;}
\]
\item[(iv)] $w(T_2, x)$ has sign-changing pattern $[-0]$ over $[0, L(b)]$, meaning 
\[\mbox{ $w(T_2,x)<0$ in $[0, h(T_2))$ and $w(T_2, h(T_2))=0$;}
\]
\item[(v)] for $t>T_2$, $w(t,x)$ has a unique nondegenerate zero $y(t)$ in $(0, L(b))$, with sign-changing pattern $[-0+]$ over $[0, L(b)]$, meaning  
\[
\mbox{$w(t,x)<0$ in $[0, y(t))$, $w(t, y(t))=0$, $w(t,x)>0$ in $(y(t), L(b)]$;}\]

\item[(vi)] $\lim_{t\nearrow T_1}x(t)=0$, $\lim_{t\searrow T_2}y(t)=L(b)$.

\end{itemize}

\end{lem}

\begin{proof} We divide the proof into several steps.

{\bf Step 1:} {\it The behavior of $w(t,x)$ for $t\in [1-\epsilon_0, T_1)$.}

We note that $w(t,x)$ satisfies
\[
w_t=w_{xx}+c(t,x) w \mbox{ for } x\in [0, k(t)], t>0
\]
with $k(t)=\min\{h(t), L(b)\}$ and some bounded function $c(t,x)$. Moreover,
$w(t,0)>0$ and $w(t, k(t))<0$ for $t\in [1-\epsilon_0, T_1)$, and for $t\in [1-\epsilon_0, 1+\epsilon_0]$,
$w(t, x)$ has a unique (nondegenerate) zero in $(0, k(t))$. Therefore we may apply Lemma 2.2 to conclude that
$w(t,x)$ has a unique zero in $(0, k(t))$ for all $t\in [1-\epsilon_0, T_1)$, and the zero is nondegenerate.
If we denote the unique zero by $x(t)$, then $t\mapsto x(t)$ is a $C^1$ function. This proves the conclusions
in part (i).

{\bf Step 2: }{\it $\lim_{t\nearrow T_1}x(t)$ and the behavior of $w(T_1, x)$.}

We next examine the limit of $x(t)$ as $t$ increases to $T_1$. If this limit does not exist, then similar to the argument in the proof of Lemma 2.4, we find that $w(T_1, x)$ is identically zero in some interval of $x$,
and it follows that $w(T_1, x)\equiv 0$ for $x\in [0, h(T_1)]$, which is impossible since $w(T_1, x)<0$ for $x$ close to $h(T_1)<L(b)$. Therefore $x_0:=\lim_{t\nearrow T_1}x(t)$ exists and $w(T_1, x_0)=0$. We necessarily have $x_0<h(T_1)$ since $w(T_1, h(T_1))<0$.

We now show that $x_0=0$. Otherwise $x_0>0$ and we may apply the maximum principle over the region
$Q_1:=\{(t,x): 0<x<x(t), 1< t\leq T_1\}$ to conclude that $w(t,x)>0$ in $Q_1$. Since $w(T_1, 0)=0$ by the definition of $T_1$, we can use the Hopf lemma to conclude that $w_x(T_1, 0)>0$. Since $V_b'(0)=0$, this implies that
$u_x(T_1, 0)>0$, which is impossible since $u(T_1,x)$ is a smooth even function of $x$ by \eqref{special-u}.
This contradiction proves $x_0=0$ and the first part of (vi) is proved.

Applying the maximum principle to $w$ over the region $Q_2:=\{(t,x): x(t)<x\leq h(t), 1<t\leq T_1\}$ and we easily see that $w<0$ in $Q_2$. In particular, $w(T_1, x)<0$ for $x\in (0, h(T_1)]$. Since we already know $w(T_1, 0)=0$,
part (ii) is proved.

{\bf Step 3:} {\it The case $t\in (T_1, T_2)$ and $t=T_2$.}

Let $Q_3:=\{(t,x): -h(t)\leq x\leq h(t), T_1<t<T_2\}$. We find that $w<0$ on the parabolic boundary of $Q_3$ except at $(T_1, 0)$. Hence we can apply the maximum principle to conclude that $w<0$ in $Q_3$, and $w(T_2, x)<0$
for $x\in (-h(T_2), h(T_2))$. By the definition of $T_2$, we see that $w(T_2, h(T_2))=0$. This proves part (iii) and part (iv).

{\bf Step 4:} {\it The case $t\in (T_2, T_2+\epsilon)$.}

Continuing from the last paragraph, we may apply the Hopf boundary lemma to conclude that $w_x(T_2, h(T_2))>0$. By continuity,
there exists $\epsilon_1>0$ small so that $w_x(t, x)>0$ for $t\in (T_2, T_2+\epsilon_1]$ and $x\in [h(T_2)-\epsilon_1, h(T_2)]=[L(b)-\epsilon_1, L(b)]$. Since $h(t)>h(T_2)=L(b)$ for $t\in (T_2, T_2+\epsilon_1]$,
we find that $w(t, L(b))>0$ for such $t$. From $w(T_2, L(b)-\epsilon_1)<0$, by continuity, we can find $\epsilon_2\in (0,\epsilon_1]$ small so that $w(t, L(b)-\epsilon_1)<0$ for $t\in (T_2, T_2+\epsilon_2]$. Thus for fixed $t\in (T_2, T_2+\epsilon_2]$, the strictly increasing function $w(t, x)$ over $[L(b)-\epsilon_1, L(b)]$ has a unique zero $y(t)\in (L(b)-\epsilon_1, L(b))$, and the zero is nondegenerate. Hence $y(t)$ is a $C^1$ function  for $t\in (T_2, T_2+\epsilon_2]$. 

We now examine the limit of $y(t)$ as $t$ decreases to $T_2$. Since $w(T_2,x)<0$ in $[0, L(b))$ and $w(T_2, L(b))=0$, the limit necessarily exists and has value $L(b)$. This proves the second part of (vi).

Since $w(T_2, 0)<0$, by shrinking $\epsilon_2$ further we may assume that $w(t, 0)<0$ for $t\in [T_2, T_2+\epsilon_2]$. Applying the maximum principle to $w$ over $Q_4:=\{(t,x): 0\leq x<y(t), T_2\leq t\leq T_2+\epsilon_2\}$, we see that $w<0$ in $Q_4$. Therefore, for $t\in (T_2, T_2+\epsilon_2]$,  $y(t)\in (0, L(b))$ is the unique zero of $w(t,x)$ 
over $[0, L(b)]$ and it is nondegenerate.

{\bf Step 5:} {\it The case $t>T_2$.}

We claim that $w(t, 0)<0$ for all $t>T_2$. Otherwise there exists $t_1>T_2$ such that $w(t,0)<0$ for $t\in [T_2, t_1)$ and $w(t_1, 0)=0$. By Lemma 2.2, the fact that $w(t,0)<0<w(t, L(b))$ for $t\in (T_2, t_1)$, and the existence of $y(t)$ for $t\in (T_2, T_2+\epsilon)$, we see that 
 $w(t,x)$ has a unique nondegenerate zero over $[0, L(b)]$ for every $t\in (T_2, t_1)$. Hence $y(t)$ can be extended to all $t\in (T_2, t_1)$. 

Let us look at the limit of $y(t)$ as $t$ increases to $t_1$. If the limit does not exist, then as before we deduce $w(t_1, x)\equiv 0$ over $[0, L(b)]$, which is impossible since $w(t_1, L(b))>0$. Hence the limit exists, and we denote it by $y(t_1)$. Clearly $w(t_1, y(t_1))=0$, which implies $y(t_1)<L(b)$.

 If $y(t_1)=0$, then applying the maximum principle to $w$ over $\{(t,x): y(t)<x<L(b), T_2<t\leq t_1\}$ we obtain $w(t_1, x)>0$ for $x\in (y(t_1), L(b))=(0, L(b))$.
It follows that $u(t_1,x)\geq V_b(x)$ for $x\in [-L(b), L(b)]$, which implies, by the comparison principle, $u(t,x)\geq V_b(x)$ for all $t>t_1$ and $x\in [-L(b), L(b)]$, contradicting our assumption that $u(t,x)\to\theta$ as $t\to+\infty$.

If $y(t_1)\in (0, L(b))$, then
applying the maximum principle to $w$ over $Q_5:=\{(t,x): 0<x<y(t), T_2\leq t\leq t_1\}$, we see that $w<0$ in $Q_5$. By Hopf's boundary lemma we have $w_x(t_1, 0)<0$, which implies $u_x(t_1, 0)<0$. But this is a contradiction since $u(t_1,x)$ is even in $x$. Our claim is now proved.

We now use Lemma 2.2 to $w$ over the region $0<x<L(b), t>T_2$. By our earlier knowledge on the zeros of $w(t,x)$ for fixed $t\in (T_2, T_2+\epsilon_2]$, we  see that for each fixed $t>T_2$, $w(t,x)$
has at most one zero in $[0, L(b)]$. Since $w(t, 0)<0$ and $w(t, L(b))>0$ there exists at least one zero in this interval. Therefore there is a unique zero and it must be nondegenerate. Denote this zero by $y(t)$, we find that $y(t)$ is a $C^1$ function. This proves part (v) and all the conclusions in the lemma are now proved.
\end{proof}

\begin{lem}\label{T1>T2}
Suppose  $T_1>T_2$. Then
\begin{itemize}
\item[(i)] for $t\in [1, T_2)$, $w(t,x)$ has a unique nondegenerate zero $x(t)$ in $(0, h(t))$, with sign-changing pattern $[+0-]$ over $[0, h(t)]$;

\item[(ii)] $w(T_2, x)$ has a unique nondegenerate zero $x(T_2)$ in $(0, L(b))$, plus a second zero at $x=L(b)$, and it has  sign-changing pattern $[+0-0]$ over $[0, L(b)]$;

\item[(iii)] for $t\in (T_2, T_1)$, $w(t,x)$ has exactly two nondegenerate zeros $x(t)<y(t)$ in $(0, L(b))$, with sign-changing pattern $[+0-0+]$ over $[0, L(b)]$;

\item[(iv)] $w(T_1, x)$ has a unique nondegenerate zero $y(T_1)$ in $(0, L(b))$, plus a second zero at $x=0$, and it has sign-changing pattern $[0-0+]$ over $[0, L(b)]$;

\item[(v)] for $t>T_1$, $w(t,x)$ has a unique nondegenerate zero $y(t)$ in $(0, L(b))$, with sign-changing pattern 
$[-0+]$ over $[0, L(b)]$;

\item[(vi)] $x(t)$ is a $C^1$ function for $t\in [1, T_1)$  with $\lim_{t\nearrow T_1}x(t)=0$, and $y(t)$ is a $C^1$ function for $t>T_2$ with $\lim_{t\searrow T_2}y(t)=L(b)$.
\end{itemize}
\end{lem}

\begin{proof} For  clarity we again break the proof into several steps.

{\bf Step 1:} {\it The case $t\in [1, T_2)$.}\;\;

The proof of part (i) is the same to that for Lemma \ref{T1<T2}. So we have a unique nondegenerate zero of $w(t,x)$ for $t\in [1, T_2)$, denoted by $x(t)$.

{\bf Step 2:} {\it  $\lim_{t\nearrow T_2}x(t)$ and the behavior of $w(T_2, x)$.}\;\;

We show that $\lim_{t\nearrow T_2}x(t)$ exists and belongs to $(0, L(b))$. If the limit does not exist, then as before we deduce $w(T_2, x)\equiv 0$ in $[0, L(b)]$, which is impossible since $w(T_2,0)>0$ by the assumption  $T_2<T_1$. Thus the limit exists and we denote it by $x(T_2)$. 

We next prove that $x(T_2)\in (0, L(b))$. Since $w(T_2, x(T_2))=0$ and $T_2<T_1$, we necessarily have $x(T_2)>0$. We show now $x(T_2)=L(b)$ leads to a contradiction. Indeed, in such a case, we can apply the maximum principle to $w$ over the region $\{(t,x): 0\leq x<x(t), 1<t\leq T_2\}$ to see that $w>0$ in this region.
In particular, $w(T_2,x)>0$ for $x\in [0, x(T_2))=[0, L(b))$. It follows that
\[
u(T_2,x)\geq V_b(x) \mbox{ for } x\in [-L(b), L(b)].
\]
Using the comparision principle we deduce $u(t,x)\geq V_b(x)$ for $x\in [-L(b), L(b)]$ and all $t>T_2$, which contradicts the assumption that $u(t,x)\to\theta$ as $t\to+\infty$. Thus we have proved that
\[
x(T_2)=\lim_{t\nearrow T_2}x(t)\in (0, L(b)).
\]
Moreover, applying the maximum principle to $w$ over $\{(t,x): 0\leq x<x(t), 1<t\leq T_2\}$ we deduce $w>0$ in this region and hence $w(T_2,x)>0$ for $x\in [0, x(T_2))$. Similarly, using the maximum principle to $w$ over
$\{(t,x): x(t)<x<h(t), 1<t\leq T_2\}$ we deduce  $w(T_2,x)<0$ for $x\in (x(T_2), L(b))$. Clearly
$w(T_2, L(b))=w(T_2, h(T_2))=0$. We have thus proved part (ii) except that we still have to show that $x(T_2)$ is a nondegenerate zero.

{\bf Step 3:} {\it The case $ t\in (T_2, T_2+\epsilon)$}.\;\;

Conitinuing from the last paragraph, we may use the Hopf boundary lemma to obtain $w_x(T_2, L(b))>0$.
Thus we can now argue as in the proof of Lemma \ref{T1<T2}  to find $\epsilon_1>0$ and $\epsilon_2\in (0,\epsilon_1]$ such that for each $t\in (T_2, T_2+\epsilon_2]$, $w(t,x)$ over $[L(b)-\epsilon_1, L(b)]$ has a unique nondegenerate zero $y(t)\in (L(b)-\epsilon_1, L(b))$, whose sign-changing pattern over $[L(b)-\epsilon_1, L(b)]$ is $[-0+]$. 

Applying Lemma 2.2 to $w$ over $\{(t,x): 0<x<\min\{h(t), L(b)\}-\epsilon_1, t\in (1, T_2+\epsilon_2)\} $(and we may shrink $\epsilon_1$ and $\epsilon_2$ to guarantee that $x(t)<r(t):=\min\{h(t), L(b)\}-\epsilon_1$ for $t\in (1, T_2+\epsilon_2)\subset (1, T_1)$), we find that $w(t,x)$ can have at most one zero in $(0, r(t))$ for $t\in (1, T_2+\epsilon_2)$. On the other hand, since $w(t,0)>0>w(t, r(t))$ for such $t$,  there exists at least one zero. Therefore there exists exactly one zero and it is nondegenerate. In other words, the nondegenerate zero $x(t)$ can be extended smoothly to $t\in [1, T_2+\epsilon)$, and $x(t)<L(b)-\epsilon_1<y(t)$ for $t\in(T_2, T_2+\epsilon_2)$. This in particular proves the remaining part of (ii).

We now consider the limt of $y(t)$ as $t$ decreases to $T_2$. Since $y(t)\in (L(b)-\epsilon_1, L(b))$ for $t\in (T_2, T_2+\epsilon_2)$ and  $w(T_2,x)<0$ for $x\in [L(b)-\epsilon_1, L(b))$, we necessarily have $\lim_{t\searrow T_2}y(t)=L(b)$. 

{\bf Step 4:} {\it The case $t\in (T_2, T_1)$.}\;\;

In view of Lemma 2.2 and the fact that
$w(t,0)>0, w(t, L(b))>0$ for $t\in (T_1, T_2)$, and that for $t\in (T_2, T_2+\epsilon_2)$, $w(t,x)$ has exactly two nondegenerate zeros $x(t)<y(t)$ over the interval $[0, L(b)]$, we see that $w(t,x)$ has at most two zeros over this interval for every $t\in (T_1, T_2)$.

We claim that for every $t\in (T_2, T_1)$, $w(t,x)$ has  exactly two nondegenerate zeros in $(0, L(b))$. Indeed, by the implicit function theorem, the nondegenerate zeros $x(t)$ and $y(t)$ can be continued to larger $t$ as long as they stay nondegenerate.  Define
\[
T^*:=\sup\{s\in (T_2, T_1): x(t) \mbox{ and $y(t)$ are nondegenerate zeros of $w(t,x)$ for every $t\in (T_2, s)$}\}.
\]
Clearly $T^*\geq T_2+\epsilon_2$. If $T^*=T_1$, then $x(t)$ and $y(t)$ are two nondegenerate zeros of $w(t,x)$ 
in $[0, L(b)]$ for
every $t\in (T_2, T_1)$. As there can be at most two zeros, our claim is proved. If $T^*<T_1$, we show that a contradiction arises. In such a case, let us consider the limit of $x(t)$ and $y(t)$ as $t$ increases to $T^*$.
Both limits must exist for otherwise $w(T^*,x)$ would be identically zero for some interval of $x$ in $[0,L(b)]$, contradicting the fact that it has at most two zeros. Let $x(T^*)$ and $y(T^*)$ denote these limits, respectively.
Then $w(T^*, x(T^*))=w(T^*, y(T^*))=0$ and hence $0<x(T^*)\leq y(T^*)<L(b)$. 
Using the maximum principle to $w$ over $\{(t,x): 0\leq x<x(t), T_2<t\leq T^*\}$, we deduce that $w(T^*,x)>0$
for $x\in [0, x(T^*))$. Similarly, $w(T^*, x)>0$ for $x\in (y(T^*), L(b)]$. 

If $x(T^*)=y(T^*)$, we have
$w(T^*, x)\geq 0$ for $x\in [0, L(b)]$. It follows that $u(T^*,x)\geq V_b(x)$ for $x\in [-L(b), L(b)]$. As before we can use the comparison principle to deduce that $u(t,x)\geq V_b(x)$ for all $t>T^*$ and $x\in [-L(b), L(b)]$, contradicting our assumption that $u(t,x)\to\theta$ as $t\to\infty$.

If $x(T^*)<y(T^*)$, then we can apply the maximum principle to $w$ over $\{(t,x): x(t)<x<y(t),\; t\in (T_2, T^*]\}$ to deduce that $w(t,x)<0$ for $x\in (x(T^*), y(T^*))$. Set $x_0:=[x(T^*)+y(T^*)]/2$. Then $w(T^*,x_0)<0$ and by continuity we can find $\epsilon_3>0$ small such that $w(t, x_0)<0$ for $t\in (T^*-\epsilon_3, T^*+\epsilon_3)$. 
We may now use Lemma 2.2 to $w$ for $0<x<x_0$ and $t\in (T^*-\epsilon_3, T^*+\epsilon_3)$
to conclude that $w(t,x)$ has at most one zero in $[0,x_0]$ for every such $t$. On the other hand from
$w(t,0)>0>w(t,x_0)$ we see that it has at least one zero. Therefore it has a unique nondegenerate zero
in $[0,x_0]$ for every $t\in (T^*-\epsilon_3, T^*+\epsilon_3)$. This implies that $x(t)$ remains a nondegenerate zero of $w(t,x)$ for $t\in [1, T^*+\epsilon_3)$. Similarly we can apply Lemma 2.2 to $w$ for $x_0<x<L(b)$ and $t\in 
(T^*-\epsilon_3, T^*+\epsilon_3)$ to see that $y(t)$ remains nondegenerate for every $t\in (T_2, T^*+\epsilon_3)$.
But this contradicts the definition of $T^*$. Our claim is thus proved. Moreover, the discussion above also shows that the sign-changing pattern of $w(t,x)$ over $[0, L(b)]$ for every $t\in (T_2, T_1)$ is $[+0-0+]$. This proves part (iii).

{\bf Step 5:} {\it $\lim_{t\nearrow T_1}x(t)$, $\lim_{t\nearrow T_1}y(t)$ and the behavior of $w(T_1,x)$.}

We first observe that both limits exist for otherwise, as before, we can deduce $w(T_1, x)\equiv 0$ for $x\in [0, L(b)]$, which is impossible since $w(T_1, L(b))>0$. Denote the two limits by $x(T_1)$ and $y(T_1)$, respectively.
Then necessarily $w(T_1, x(T_1))=w(T_1, y(T_1))=0$ and $0\leq x(T_1)\leq y(T_1)<L(b)$. By the same argument used in Step 2 of the proof of Lemma \ref{T1<T2}, we find that $x(T_1)=0$. We show next that $y(T_1)>0$.
Indeed, if $y(T_1)=0$, then we can use the maximum principle to $w$ over $\{(t,x): y(t)<x<L(b), t\in (T_2, T_1]\}$ to deduce that $w(T_1, x)>0$ in $(0, L(b)]$. It follows that $u(T_1,x)\geq V_b(x)$ for $x\in [-L(b), L(b)]$, which implies $u(t,x)\geq V_b(x)$ for all $t>T_1$ and $x\in [-L(b), L(b)]$, a contradiction to the assumption that $u(t,x)\to\theta$ as $t\to+\infty$. We have thus proved that
\[
0=x(T_1)<y(T_1)<L(b).
\]
Applying the maximum principle to $w$ over $\{(t,x): x(t)<x<y(t),\; t\in(T_2, T_1]\}$ we deduce $w(T_1, x)<0$ for 
$x\in (0, y(T_1))$. Similarly using the maximum principle to $w$ over $\{(t,x): y(t)<x<L(b), t\in (T_2, T_1]\}$, we obtain $w(T_1, x)>0$ for $x\in (y(T_1), L(b)]$. Thus $w(T_1, x)$ has sign-changing pattern $[0-0+]$ over $[0, L(b)]$. This proves part (iv) except that we still have to show that $y(T_1)$ is a nondegenerate zero.

{\bf Step 6:} {\it The case $t\in [T_1, T_1+\epsilon)$.}

Fix $y_0\in (0, y(T_1))$. We have $w(T_1, y_0)<0$. By continuity there exists $\epsilon>0$ small such that
$w(t, y_0)<0$ for $t\in (T_1-\epsilon, T_1+\epsilon)$. We now apply Lemma 2.2 to $w(t,x)$ over the region
$y_0<x<L(b)$ and $t\in (T_1-\epsilon, T_1+\epsilon)$. Since $w(t,y_0)<0<w(t, L(b))$ for such $t$, $w(t,x)$ has at least one zero in $(y_0, L(b))$. But for $t\in (T_1-\epsilon, T_1)$, we already know that there is a unique zero.
Hence Lemma 2.2 infers that $w(t,x)$ has a unique zero in $(y_0, L(b))$ and it is nondegenerate for every $t
\in (T_1-\epsilon, T_1+\epsilon)$. This implies that $y(t)$ can be extended to all $t\in (T_2, T_1+\epsilon)$ as a nondegenerate zero of $w(t,x)$. In particular, $y(T_1)$ is a nondegenerate zero of $w(T_1, x)$. We have now proved all the conclusions in part (iv) of the lemma.

 Using the maximum principle to $w$ over $\{(t,x): -y(t)<x<y(t), T_1<t<T_1+\epsilon\}$,
we find that $w<0$ in this region. In particular, $w(t,x)<0$ for $x\in [0, y(t))$ and $t\in (T_1, T_1+\epsilon)$.
Hence $y(t)$ is the unique nondegenerate zero of $w(t,x)$ in $[0, L(b)]$ for every $t\in (T_1, T_1+\epsilon)$.
 
{\bf Step 7:} {\it The case $t>T_1$.}

From the discussions in Step 6 above, we already know that $w(t,0)<0$ for $t\in (T_1, T_1+\epsilon)$.
The same argument in Step 5 of the proof of Lemma \ref{T1<T2} shows that $w(t,0)<0$ for all $t>T_1$.
We may now use Lemma 2.2 to $w$ over the region $0<x<L(b)$ and $t>T_1$ in the same way as in Step 5  of Lemma \ref{T1<T2} to conclude that, for every $t>T_1$, $w(t,x)$ has a unique nondegenerate zero $y(t)$ in $(0, L(b))$, and $w(t,x)$ has sign-changing pattern $[-0+]$ over $[0, L(b)]$. This proves part (v) of the lemma.

Since $y(t)$ is nondegenerate, it is a $C^1$ function for $t>T_1$. 
Let us note that the other conclusions in part (vi) have already been proved in Steps 3, 4 and 5. The proof of the lemma is now complete.
\end{proof}

\begin{lem}\label{T1=T2}
Suppose  $T_1=T_2$. Then
\begin{itemize}
\item[(i)] for $t\in [1, T_1)$, $w(t,x)$ has a unique nondegenerate zero $x(t)$ in $(0, h(t))$, with sign-changing pattern $[+0-]$ over $[0, h(t)]$;

\item[(ii)] $w(T_1,x)$ has sign-changing pattern $[0-0]$ over $[0, L(b)]$;

\item[(iii)] for $t>T_1=T_2$, $w(t,x)$ has a unique nondegenerate zero $y(t)$ in $(0, L(b))$, with sign-changing pattern $[-0+]$ over $[0, L(b)]$;

\item[(iv)] $\lim_{t\nearrow T_1}x(t)=0$, $\lim_{t\searrow T_1}y(t)=L(b)$.

\end{itemize}

\end{lem}

\begin{proof}
This follows from simple variations of the proof of Lemma \ref{T1<T2}. The proof of part (i) is the same as Step 1 of the proof of Lemma \ref{T1<T2}. 

Step 2 there requires variations. If  $\lim_{t\nearrow T_1}x(t)$ does not exist, then as before we obtain $w(T_1, x)\equiv 0$ for $x\in [0, L(b)]$. Since this time $w(T_1, h(T_1))=0$, we derive a contradiction in a different way as follows. From the above identity we obtain $u(T_1,x)\equiv V_b(x)$, which implies, by the comparison principle, $u(t,x)\geq V_b(x)$ for all $t>T_1$ and $x\in [-L(b), L(b)]$. But this contradicts our assumption that $u(t,x)\to\theta$ as $t\to +\infty$. Therefore $x_0:=\lim_{t\nearrow T_1}x(t)$ exists. If $x_0\in (0, L(b)]$, then we can derive a contradiction as in Step 2 of Lemma \ref{T1<T2}. Thus $x_0=0$ and the first part of (iv) is proved. 

We can also obtain $w(T_1, x)<0$ for $x\in (0, L(b))$ as in Step 2 of Lemma \ref{T1<T2}. Since $T_1=T_2$, clearly $w(T_1, L(b))=0$. This proves part (ii).

By the argument in Step 4 of Lemma \ref{T1<T2}, we can find small $\epsilon_1>0$ and $\epsilon_2\in (0,\epsilon_1]$ such that for every $t\in (T_1,T_1+\epsilon_2]$, $w(t,x)$
over $[L(b)-\epsilon_1, L(b)]$  has a unique nondegenerate zero $y(t)\in (L(b)-\epsilon_1, L(b))$. Moreover, 
$\lim_{t\searrow T_1}y(t)=L(b)$. 

Using the maximum principle to $w$ over $\{(t,x): -y(t)<x<y(t), T_1<t\leq T_1+\epsilon_2\}$ we see that
$w<0$ in this region. In particular, $w(t,x)<0$ for $x\in [0, y(t))$ and $t\in (T_1, T_1+\epsilon_2]$.
Hence $y(t)$ is the unique zero of $w(t,x)$ over $[0, L(b)]$ for every $t\in (T_1, T_1+\epsilon_2]$. 

We may now follow Step 5 of Lemma \ref{T1<T2} to complete the proof.
\end{proof}

Let us note that Lemmas \ref{T1<T2}, \ref{T1>T2} and \ref{T1=T2} give a complete classification of the
sign-changing  patterns of the function $w_b(t,x):=u(t,x)-V_b(x)$, for every fixed $b\in (0, \delta_0)$ and $t\geq 1$.
We next use this information to prove \eqref{assump}.

\subsubsection{\bf Completion of the proof of \eqref{assump}.}

Denote 
\[
m(\delta):=\sup_{b\in (0,\delta]}\frac{l(b)}{L(b)}.
\]
By Lemma \ref{lemm-key} we find that $m(\delta)$ decreases to $0$ as $\delta\to 0$. Therefore, \eqref{assump}
is a direct consequence of the following result.

\begin{prop}\label{theta(t)}
For any given $\delta\in (0,\delta_0)$, there exists $M=M_\delta>0$ such that $\theta(t)/h(t)\leq m(\delta)$ for $t\geq M$.
\end{prop}

\begin{proof}
Let $\delta\in (0,\delta_0)$ be arbitrarily given. Since $L(b)$ is continuous in $b$ and $L(b)\to+\infty$ as $b\to 0$, 
for any given $h\in (L(\delta), +\infty)$, there exists $b\in (0,\delta)$ such that $L(b)=h$. We now choose $M_1>0$ such that $h(t)>L(\delta)$ for $t\geq M_1$. Then we can find $b(t)\in (0, \delta)$ such that \footnote{ $b(t)$ need not be unique, nor continuous with respect to $t$. We may require $b(t)$ to be the minimal solution $b$ to $h(t)=L(b)$ to make $b(t)$ uniquely determined.}
\[
h(t)=L(b(t)) \mbox{ for } t\geq M_1.
\]
Since $h(t)\to +\infty$ as $t\to+\infty$, we necessarily have $b(t)\to 0$ as $t\to+\infty$.

We define
\[
\Sigma_1:=\{t\geq M_1: \theta(t)\leq l(b(t))\},\; \Sigma_2:=\{t\geq M_1: \theta(t)>l(b(t))\}.
\]
For $t\in \Sigma_1$, since $b(t)<\delta$, clearly
\[
\frac{\theta(t)}{h(t)}=\frac{\theta(t)}{L(b(t))}\leq \frac{l(b(t))}{L(b(t))}\leq m(\delta).
\]
If there exists $M_2>M_1$ such that $\Sigma_2\cap [M_2, +\infty)=\emptyset$, then we can take $M=M_2$ and our proof is complete.

It remains to consider the case that $\Sigma_2\cap[n,+\infty)\not=\emptyset$ for every $n\geq 1$.
Let $M_2>M_1$ be chosen such that, for every $t\in\Sigma_2\cap[M_2, +\infty)$,
\begin{equation}
\label{M2}
u(t,0)<\theta+\delta \mbox{ and }
\theta(t)>l(\delta).
\end{equation}
The second inequality is possible since for $t\in\Sigma_2$, $\theta(t)>l(b(t))\to+\infty$ as $t\to+\infty$.

We now fix $t\in \Sigma_2\cap [M_2, +\infty)$, and denote
\[
b_0=b(t),\; b_1=u(t,0)-\theta.
\]
To stress the $b$-dependence of $T_1$ and $T_2$ determined by $w_b(t,x)$ in section 4.3.2, we will write
\[
T_1=T_1^b,\; T_2=T_2^b.
\]

{\bf Claim:} $b_0<b_1<\delta$.
\ \ The second inequality follows directly from \eqref{M2}. We next prove the first by examining the sign-changing pattern of $w_{b_0}(t,x)$. Since $h(t)=L(b_0)$, we are in the case $t=T_2^{b_0}$ considered in section 4.3.2. By Lemmas \ref{T1<T2}-\ref{T1=T2}, we find that $w_{b_0}(t,x)<0$ for $x<L(b_0)$ but close to $L(b_0)$. On the other hand, from $\theta(t)>l(b_0)$ we find $w_{b_0}(t, l(b_0))>0$. Hence $w_{b_0}(t,x)$ has a zero in $(l(b_0), L(b_0))$. Such a situation can only happen in the case described by Lemma \ref{T1>T2} part (ii), where $w_{b_0}(t,x)$ has sign-changing pattern $[+0-0]$ over $[0, L(b_0)]$. Therefore $w_{b_0}(t,0)=b_1-b_0>0$. This proves the claim.

The discussions below are organized according to the  three cases: $l(b_1)>\theta(t)$, $\l(b_1)=\theta(t)$ and $l(b_1)<\theta(t)$.

{\bf The case ``$l(b_1)>\theta(t)$''.} \ \  Since $l(\delta)<\theta(t)$ (see \eqref{M2}), by continuity there exists $b_2\in (b_1,\delta)$ such that $l(b_2)=\theta(t)$. We next prove $h(t)>L(b_2)$ by examining the sign-changing pattern of $w_{b_2}(t,x)$.
From $b_2>b_1$ we obtain $w_{b_2}(t,0)=b_1-b_2<0$. Hence we are in the case $t>T_1^{b_2}$. Moreover, $l(b_2)=\theta(t)$ implies $w_{b_2}(t, l(b_2))=0$. Therefore we can only be in the case described in Lemma \ref{T1<T2} part (v), or Lemma \ref{T1>T2} part (v), or Lemma \ref{T1=T2} part (iii). In all these cases we have $t>T_2^{b_2}$. Hence $h(t)>L(b_2)$. It follows that
\[
\frac{\theta(t)}{h(t)}=\frac{l(b_2)}{h(t)}<\frac{l(b_2)}{L(b_2)}\leq m(\delta).
\]

{\bf The case ``$l(b_1)=\theta(t)$''.} \ \  We examine the sign-changing pattern of $w_{b_1}(t,x)$. Clearly
$w_{b_1}(t,0)=0$. So we are in the situation that $t=T_1^{b_1}$. Moreover, $w_{b_1}(t,x)$ has a zero at $x=\theta(t)=l(b_1)\in (0, L(b_1))$. This is possible only in the case of Lemma \ref{T1>T2} part (iv), which indicates that $t>T_2^{b_1}$. Hence $h(t)>L(b_1)$ and we have
\[
\frac{\theta(t)}{h(t)}=\frac{l(b_1)}{h(t)}<\frac{l(b_1)}{L(b_1)}\leq m(\delta).
\]

{\bf The case ``$l(b_1)<\theta(t)$''.} This is the most complicated 
case to handle. We first examine the sign-changing pattern of $w_{b_1}(t,x)$. Due to $w_{b_1}(t,0)=0$ we are in the case $t=T_1^{b_1}$. By Lemmas \ref{T1<T2}-\ref{T1=T2}, we find that $w_{b_1}(t,x)<0$ for $x>0$ but close to $0$. On the other hand, $l(b_1)<\theta(t)$ implies that $w_{b_1}(t, l(b_1))>0$. Therefore $w_{b_1}(t,x)$ has a zero in $(0, l(b_1))$. This is possible only in the case of Lemma \ref{T1>T2} part (iv), where $t>T_2^{b_1}$ and $w_{b_1}(t,x)$ has sign-changing pattern $[0-0+]$ over $[0, L(b_1)]$ with a unique nondegenerate zero $y(t)\in (0, L(b_1))$. As $w_{b_1}(t,l(b_1))
>0$, we necessarily have $y(t)<l(b_1)$.

Since $l(b_1)<\theta(t)$ and $l(b)\to+\infty$ as $b\to 0$, by decreasing $b$ from $b_1$ we can find  $b_3\in(0, b_1)$ such that
\begin{equation}\label{b3}\mbox{
 $l(b_3)=\theta(t)$ and $l(b)<\theta(t)$ for $b\in (b_3, b_1]$.}
\end{equation}
 We want to show that $L(b_3)<h(t)$. If this is proved, then as before we have
\[
\frac{\theta(t)}{h(t)}=\frac{l(b_3)}{h(t)}<\frac{l(b_3)}{L(b_3)}\leq m(\delta).
\]
Thus we can take $M=M_2$ and the proof of the proposition is complete.

It remains to prove 
\[
L(b_3)<h(t).
\]
 Since $h(t)>L(b_1)$, by continuity, for $b<b_1$ but close to $b_1$ we still have
$h(t)>L(b)$. For such $b$ clearly $w_b(t,0)=b_1-b>0$, and hence  $t\in(T_2^b,T_1^b)$. We now examine the sign-changing profile of $w_b(t,x)$.
By Lemma \ref{T1>T2} part (iii), $w_b(t,x)$ has sign-changing pattern $[+0-0+]$ over $[0, L(b)]$ with exactly two nondegenerate zeros $x(t)<y(t)$ in $ (0, L(b))$. Let us also note that from $y(t)<l(b_1)$ we obtain $y(t)<l(b)$ for such $b$.

As we are now varying $b$ while keeping $t$ fixed, it is convenient to regard $w_b(t,x)$ as a function of $(b,x)$, and write
\[
W(b,x)=w_b(t,x).
\]
Similarly, we will write $X(b)=x(t)$ and $Y(b)=y(t)$ to stress their dependence on $b$. Thus $W(b,x)$ has sign-changing pattern $[+0-0+]$ over $[0, L(b)]$ with exactly two nondegenerate zeros $X(b)<Y(b)$ for  $b<b_1$
but close to $b_1$.

By the continuous dependence of $V_b(x)$ on $b$, we find that $W(b,x)$, $W_x(b,x)$  are continuous in $b$ uniformly for $x$
in bounded sets.
By the implicit function theorem we find that $X(b)$ and $Y(b)$ are $C^1$ functions of $b$ as long as they are nondegenerate zeros of $W(b,x)$. 
Hence by our analysis above, $X(b)$ and $Y(b)$ are defined for all $b<b_1$ and close to $b_1$, and they are $C^1$ functions of $b$.  We now consider the functions $X(b)$, $Y(b)$ and $L(b)$ as $b$ is decreased further. We claim that for all $b\in [b_3, b_1)$, $X(b)$ and $Y(b)$ are defined and
\[
0<X(b)<Y(b)<L(b)<h(t).
\]
Set 
\[
\Lambda:=\big\{c\in (0, b_1): X(b), Y(b) \mbox{ are nondegenerate zeros of $W(b,x)$}\]
\[
\hspace{4cm} \mbox{ and $0<X(b)<Y(b)<L(b)<h(t)$ for all } b\in[c, b_1)\big\},
\]
and 
\[
b_*=\inf \Lambda.
\]
By our earlier discussion, we have $b_*<b_1$.
We prove that 
\[
b_*< b_3.
\]
 Suppose on the contrary $b_*\in [ b_3, b_1)$. Let us first examine the limits of $X(b)$ and $Y(b)$ as $b$ decreases to $b_*$. If at least one of these limits does not exist, then $W(b_*,x)=w_{b_*}(t,x)$ would be identically zero for $x$ in some interval contained in $[0, L(b_*)]$. But by Lemmas \ref{T1<T2}-\ref{T1=T2}, no function $w_b(t,x)$ with fixed $t\geq 1$ and $b\in (0,\delta_0)$ can have such a sign-changing pattern.  Hence both limits exist, and we denote them by $X(b_*)$ and $Y(b_*)$, respectively.

Evidently $W(b_*, X(b_*))=W(b_*, Y(b_*))=0$. Since $W(b_*,0)=b_1-b_*>0$, we necessarily have
\[
0<X(b_*)\leq Y(b_*)\leq L(b_*)\leq h(t).
\]
We claim that 
\[
Y(b_*)\leq l(b_*). 
\]
Otherwise, $Y(b_*)>l(b_*)$,
and in view of $Y(b)<l(b)$ for $b<b_1$ but close to $b_1$, we can find $b_4\in (b_*, b_1)$ such that $Y(b_4)=l(b_4)$, which implies that $u(t, l(b_4))=\theta$ and hence $\theta(t)=l(b_4)$, contradicting \eqref{b3}. This proves our claim.

From $W(b_*,0)>0$ and $L(b_*)\leq h(t)$ we see that $T_2^{b_*}\leq t<T_1^{b_*}$.
Moreover, for each $b\in (b_*, b_1)\subset\Lambda$, we have $T_2^b<t<T_1^b$ and by Lemma \ref{T1>T2}, $W(b,x)$ has sign-changing pattern $[+0-0+]$. It follows that 
\begin{equation}
\label{b*}
W(b_*,x)\geq 0 \mbox{ for } x\in [Y(b_*), L(b_*)].
\end{equation}

We show next that $L(b_*)<h(t)$. Otherwise $L(b_*)=h(t)$ and hence $t=T_2^{b_*}<T_1^{b_*}$. Therefore
we can use Lemma \ref{T1>T2} part (ii) to conclude that $W(b_*, x)$ has sign-changing pattern $[+0-0]$, which
contradicts \eqref{b*}. This proves $L(b_*)<h(t)$ and hence $t\in (T_2^{b_*}, T_1^{b_*})$. But then we can apply
Lemma \ref{T1>T2} part (iii) to conclude that $W(b_*,x)$ has sign-changing pattern$[+0-0+]$ over $[0, L(b_*)]$, with exactly two nondegenerate zeros. This implies that $b_*\in\Lambda$ and by the implicit function theorem,
any $b<b_*$ and close to $b_*$ also belongs to $\Lambda$, contradicting the definition of $b_*$. We have thus proved $b_*<b_3$ and so $b_3\in\Lambda$. In particular, $L(b_3)<h(t)$, as we wanted.
\end{proof}

\end{document}